\documentclass{article}

\usepackage[backend=biber,maxbibnames=4,sorting=none]{biblatex}

\usepackage{authblk}

\usepackage{xspace}
\usepackage{siunitx}
\usepackage{csvsimple}
\usepackage{todonotes}
\usepackage{mathtools,dsfont}
\usepackage{amsmath,amssymb,amsthm}
\usepackage{mathabx}

\usepackage{thmtools}
\usepackage{thm-restate}

\usepackage[norelsize,ruled,vlined,commentsnumbered]{algorithm2e}
\usepackage{tikz}
\usetikzlibrary{matrix, positioning, patterns, shapes.multipart, arrows.meta}
\usepackage{ circuitikz }

\usepackage{pgfplots}
\usepackage{tikzscale}
\usepackage{adjustbox}
\usepackage{subcaption}
\usepackage[group-digits=integer,group-minimum-digits=4,print-unity-mantissa=false]{siunitx}
\usepackage{booktabs}
\usepackage{multicol}
\usepackage{placeins}
\usepackage{enumitem}

\usepackage[short]{optidef}

\usetikzlibrary{external}
\graphicspath{{Figures/}}

\usepackage[hidelinks]{hyperref}

\newcommand{\st}{\textrm{s.t.\xspace}}
\newcommand{\ie}{i.e.,\xspace}

\newcommand{\define}{\coloneqq}

\newcommand{\Real}{\mathds{R}}
\newcommand{\Nat}{\mathds{N}}
\newcommand{\FS}{\mathcal{C}}

\newcommand{\Lag}{\mathcal{L}}

\newcommand{\IndexSet}{\mathcal{I}}
\newcommand{\I}[1]{\IndexSet_{#1}}
\newcommand{\II}[2]{\IndexSet_{(#1,#2)}}
\newcommand{\SetN}{\mathcal{N}}
\newcommand{\VanConUP}{G}
\newcommand{\ConConUP}{H}
\newcommand{\VanCon}[2]{\VanConUP_{#1,#2}}
\newcommand{\ConCon}[1]{\ConConUP_{#1}}
\newcommand{\DVanCon}[2]{\nabla \VanCon{#1}{#2}}
\newcommand{\DConCon}[1]{\nabla \ConCon{#1}}
\newcommand{\idxVCL}[2]{#2+(#1-1)n_\VanConUP+n_\ConConUP}%{{#1}_G}
\newcommand{\idxVC}[2]{#1,#2}
\newcommand{\idxCC}[1]{#1}%{{#1}_H}

\newcommand{\ECon}{g}
\newcommand{\ICon}{h}
\newcommand{\DECon}{\nabla g}
\newcommand{\DICon}{\nabla h}

\newcommand{\primalProd}[2]{\langle #1, #2 \rangle_{X}}
\newcommand{\dualProd}[2]{\langle #1, #2 \rangle_{Y}}

\newcommand{\primalNorm}[1]{\|#1\|_{X}}
\newcommand{\dualNorm}[1]{\|#1\|_{Y}}

\theoremstyle{definition}% default
\newtheorem{theorem}{Theorem}

\newtheorem{lemma}[theorem]{Lemma}

\newtheorem{corollary}[theorem]{Corollary}
\newtheorem{definition}{Definition}

\newcommand{\IfU}[1]{\If{\textup{#1}}}

\newcommand{\PosIndices}{\mathcal{B}^\ConConUP_{+}}
\newcommand{\ZeroIndices}{\mathcal{B}^\ConConUP_{0}}
\newcommand{\PosIndicesVC}{\mathcal{B}_{+}}
\newcommand{\BetaIndices}{\mathcal{B}^\ConConUP_\mathrm{nsg}}
\newcommand{\ZeroBetaIndices}{\mathcal{B}^\ConConUP_\mathrm{sg}}
\newcommand{\BetaIndicesVC}{\mathcal{B}_\mathrm{nsg}}
\newcommand{\ZeroBetaIndicesVC}{\mathcal{B}_\mathrm{sg}}

\newcommand{\IUBIndices}{\mathcal{I}_{\mathrm{UB}}^{\ConConUP}}

\DeclareMathOperator{\MPBVCS}{MPBVCS}
\DeclareMathOperator{\BVC}{BVC}
\DeclareMathOperator{\NLP}{NLP}
\DeclareMathOperator{\super}{super}

\newcommand{\signature}{\sigma}

\renewcommand{\epsilon}{\varepsilon}

\newcommand{\cupdot}{\mathbin{\mathaccent\cdot\cup}}

\newcommand{\instName}[1]{\texttt{#1}}

\newcommand{\thetitle}{Constraint Qualifications and Gradient Flows for Block Vanishing Constraint Problems}

\newcommand{\firstauthor}{Julian Niederer}
\newcommand{\secondauthor}{Christoph Hansknecht}
\newcommand{\thirdauthor}{Andreas Potschka}
\newcommand{\fourthauthor}{Ekaterina Kostina}

\newcommand{\firstaffil}{Interdisciplinary Center for Scientific Computing, University of Heidelberg, Heidelberg, Germany}
\newcommand{\secondaffil}{Institute of Mathematics, Clausthal University of Technology, Clausthal-Zellerfeld, Germany}

\hypersetup{
  pdftitle = {\thetitle},
  pdfauthor = {\firstauthor, \secondauthor, \thirdauthor, \fourthauthor}
}

\tikzset{
  potential_bar/.style={line cap=round, line join=round, very thin},
  potential_tenbar/.style={line cap=round, line join=round, thin},
  bar/.style={line cap=round, line join=round},
  potential_coord/.style={fill,circle, inner sep=0},
  loading_force/.style={->,>=stealth, very thick, gray},
  solution_bar/.style={line join=round},
}

\newcommand{\drawNumbersPotentialBars}[1]{%
  \csvreader[head to column names,
  ]
  {Data/#1}{}{
    \draw[potential_tenbar] (\srcx, \srcy) -- (\tgtx,\tgty);

    \path let \p1 = (\srcx,\srcy), \p2 = (\tgtx,\tgty) in
      node[draw, circle, fill=white, inner sep=0.2pt, font=\small] at ($(\p1)!.5!(\p2)+ (\posx,\posy)$) {\the\numexpr\csvcoli+1\relax};
  }
}

\makeatletter
\newcommand{\setword}[2]{%
  \phantomsection
  #1\def\@currentlabel{\unexpanded{#1}}\label{#2}%
}
\makeatother

\newenvironment{enumalpha}{\begin{enumerate}[label=(\alph*)]}{\end{enumerate}}

\begin{document}

\title{\thetitle}

\author[1]{\firstauthor}
\author[2]{\secondauthor}
\author[2]{\thirdauthor}
\author[1]{\fourthauthor}

\affil[1]{\firstaffil}
\affil[2]{\secondaffil}

\maketitle
\begin{abstract}
Mathematical Programs with Blocks of Vanishing Constraints (MPBVCs) are a
challenging class of optimization problems due to the loss of standard constraint qualifications. 
Based on existing theory for classical MPVCs, i.e., MPBVC with block size one, we extend the constraint-qualification theory for MPBVCs 
and show that the known LICQ-type condition for MPBVCs implies the Guignard Constraint Qualification (GCQ). 
A novel piecewise gradient flow approach converges under GCQ. 
It directly computes strongly stationary points and is an extension of the 
piecewise gradient flow approach for classic MPVC problems. 
Numerical experiments on several truss topology design instances 
with two and even three applied load forces demonstrate the robustness and effectiveness of the method.
\end{abstract}

\section{Introduction}

In this paper we investigate so-called \emph{Mathematical Programs with Blocks of Vanishing Constraints}~\cite{mpbvc_pal,diss_pal} (MPBVCs), 
which generalize the well-known class of Mathematical Programs with Vanishing Constraints (MPVCs). 
Our primary aim is to obtain \emph{\ref{word:strong_stationary} points} for MPBVCs, which are equivalent to \emph{Karash-Kuhn-Tucker} (KKT) points. 
To achieve this, we require suitable optimality conditions together with appropriate constraint qualifications. 
However, existing theory for MPBVCs is not yet complete in this regard. 
Therefore, we state new constraint qualifications tailored for MPBVC based on existing theory for MPVC, ensuring that the Sequential Homotopy Method~\cite{sequential_homotopy} 
can be used to approximate strongly stationary points for MPBVCs.
An MPBVC is an optimization problem of the form
\begin{equation}
  \label{eq:mpbvc}
  \tag{MPBVC}
  \begin{aligned}
    \min_{x} \quad & f(x) \\
    \st \quad & \ECon(x) = 0, && \\
              & \ICon(x) \geq 0, && \\
               & \ConCon i(x) \geq 0, && \forall i\,\in \{1,\dots,n_\ConConUP\} \\
               & \VanCon ij(x)\cdot\ConCon i(x) \geq 0, && \forall (i,j)\in \{1,...,n_\ConConUP\}\times\{1,\dots,n_\VanConUP\}, \\
  \end{aligned}
\end{equation}
where $f, \VanCon ij, \ConCon i: \Real^{n} \to \Real$,
$\ECon : \Real^{n} \to \Real^{n_\ECon}$, and
$\ICon : \Real^{n} \to \Real^{n_\ICon}$ are functions which we assume
to be at least twice continuously differentiable. The constraints $\VanCon ij$ are called \emph{vanishing}
constraints~(VCs) because, depending on the value of their
corresponding \emph{controlling} constraint~(CC) $\ConCon i$,
$\VanCon ij$ is either present in the problem (when
$\ConCon{i}(x) > 0$) or vanishes from it (when $\ConCon{i}(x) = 0$).
MPVCs are MPBVCs with a block size
$n_\VanConUP$ equal to 1, \ie restricted to the case where
each controlling constraint only ever controls a single vanishing constraint.
Note that we assume that $n_\VanConUP$ is equal for each block. This
simplification is for notational convenience only and our results
hold in the non-homogeneous case as well.

As is the case for MPVCs, difficulties in solving MPBVCs arise from the loss of strong constraint qualifications.
Classical conditions such as \emph{Linear Independence Constraint Qualification (LICQ)}~\cite{numerical_optimization} and
\emph{Mangasarian-Fromovitz Constraint Qualification (MFCQ)}~\cite{mfcq} are typically not satisfied at bi-active points,
impeding the determination of suitable dual multipliers that satisfy the KKT conditions.
Furthermore the products $ \VanCon ij(x)\cdot\ConCon i(x)$ results in a non convex feasible set,
where either $\ConCon{i}(x) > 0$ and $\VanCon{i}{j}(x) \geq 0$ or $\ConCon{i}(x) = 0$
and $\VanCon{i}{j}(x)$ is free (see e.g. Figure~\ref{fig:MPBVC_vertical}). Depending on the additional constraints $g$ and $h$,
the feasible set may even become disconnected.

MPVCs can be transformed into \emph{Mathematical Programs
  with Equilibrium Constraints}, called MPECs for short~\cite{mpec,mpvc}.
As opposed to MPVCs, in MPECs the product $G_{i}H_{i}$ must be equal to zero and
each constraint $G_{i}$ and $H_{i}$ must be non-negative.
Solution strategies for MPECs are also the subject of current research, 
which relies, for example, on relaxation approaches~\cite{steffensen2010new, scholtes} or solving a sequence of 
specific subproblems~\cite{nurkanovic2025globally}. 
The problem class \emph{Mathematical Programs with Complementarity Constraints} (MPCCs)
encompasses MPVCs and MPECs, where solution strategies can be found, for example,
 in~\cite{nurkanovic2025solving,hoheisel2013theoretical}.
However it remains unclear how to reformulate the block structure of the vanishing 
constraints into an MPCC without increasing the problem size 
or introducing additional stationarity concepts. 
We therefore focus on the direct solution of MPBVCs.

In~\cite{mpvc}, a modified version of the classical constraint qualifications was introduced along with
corresponding optimality conditions for MPVC (\ie $n_\VanConUP = 1$) as well as a definition of strong stationarity (see~\cite[Definition~6.1.1]{diss_hoheisel}).
A more detailed analysis of constraint qualifications is provided in~\cite{on_the_abadie, mpvc_stat_cons}, 
where both the Abadie and Guignard conditions are examined for MPVC. The results indicate that the the Guignard constraint qualification 
holds under similarly mild (LICQ-type) assumptions. 
This LICQ-type constraint qualifications implies strong stationarity for MPVC. Note that which are KKT-points (see~\cite{mpvc}), 
and therefore supports second-order sufficient conditions and local convergence analysis \cite{relaxation_method}.

For \eqref{eq:mpbvc},~\cite{diss_pal,mpbvc_pal} extended results on MPVC
to the case of blocks of vanishing constraints (\ie $n_\VanConUP \geq 1$). 
They analyze corresponding stationarity concepts and constraint qualifications for strong and weak stationarity. 
In particular they introduced the LICQ-type (\ref{word:BVC_LICQ}) constraint qualification. 
While~\ref{word:BVC_LICQ} ensures weak stationarity, strong 
stationarity at optimal points under~\ref{word:BVC_LICQ} was not shown. 

Motivated by~\cite{on_the_abadie} and~\cite{diss_hoheisel,mpvc}, we show that~\ref{word:GCQ} for MPBVC implies strong stationarity at optimal points, 
which are equivalent to KKT points and therefore coincides with classic stationarity concepts of the original problem.
As intermediate steps, we introduce MFCQ-type (\ref{word:BVC_MFCQ}), Abadie-type (\ref{word:BVC_ACQ}), and Guignard-type (\ref{word:BVC_GCQ}) constraint qualifications, and prove that, 
together with an additional independence condition (\ref{word:partial_BVC_LICQ}), the~\ref{word:BVC_LICQ} imply~\ref{word:GCQ} in the MPBVC setting.

These theoretical results enable the computation of strongly stationary points for 
\eqref{eq:mpbvc} via the Sequential Homotopy Method~\cite{sequential_homotopy},
since under GCQ, first-order critical points coincide with equilibrium points of the employed projected antigradient / gradient flow. 
We therefore propose a practical method in Section~\ref{sec:algorithm} by extending a
variant of the sequential homotopy method used to solved
MPVCs presented in~\cite{flow_preprint} to the block vanishing case.

Primal-dual gradient flows interpret constrained optimization problems as 
saddle-point dynamical systems where equilibria coincide with KKT points.
Lyapunov-based stability analyses were developed by Feijer and Paganini~\cite{FP2010}, 
while subsequent works gave convergence properties for projected primal-dual 
dynamics~\cite{CMC2016} and exponential stability~\cite{QuLi2018}. 
Related developments include, for example, first-order primal-dual methods such as the Chambolle-Pock
algorithm~\cite{CP2011}. The sequential homotopy method~\cite{sequential_homotopy} 
is based on a projected backward Euler discretization of the projected primal-dual resp. antigradient / gradient flow, 
providing a globalization strategy for locally convergent optimization methods.

We present a numerical case study in Section~\ref{sec:num_ex} based on truss topology design problems. These problems
are of MPVC type in case of a single load case but generalize to MPBVCs under the natural extension of multiple load
cases. We compare our approach with the relaxation approach introduced in~\cite{diss_pal}, which approximates weak stationary points.

The contributions and structure of this paper are as follows:
\begin{itemize}
\item
  We begin in Section~\ref{sec:constraints_qualifications} by expanding the theory of constraint qualifications for \eqref{eq:mpbvc}.
  We prove, that~\eqref{word:BVC_LICQ} implies~\eqref{word:GCQ} at a feasible point.
  Furthermore, we show that if~\eqref{word:GCQ} holds at a local minimum, it is a strongly stationary point and therefore a KKT point.
\item
  We extend the flow-based algorithm from~\cite{flow_preprint} in Section~\ref{sec:algorithm} and show 
  it is capable of approximating strongly stationary points.
\item
  We perform several computational experiments in Section~\ref{sec:num_ex} by solving truss topology optimization instances with several load cases.
\end{itemize}

\paragraph{Notation}
For a point $x \in \Real^{n}$, we define with $\SetN_\ECon\define\{1,\dots,n_\ECon\}$, $\SetN_\ICon\define\{1,\dots,n_\ICon\}$,
$\SetN_\ConConUP\define\{1,\dots,n_\ConConUP\}$, and $\SetN_\VanConUP\define\{1,\dots,n_\VanConUP\}$ the 
index sets
\begin{align}\label{vanishing: indexsets}
\begin{split}    
\II{+}{+}(x) &\define\{(i,j)\in \SetN_\ConConUP\times\SetN_\VanConUP\mid\ConCon i(x)>0, \, \VanCon ij(x) >0\},\\
\II{+}{0}(x) &\define\{(i,j)\in \SetN_\ConConUP\times\SetN_\VanConUP\mid\ConCon i(x)>0, \, \VanCon ij(x) =0\},\\
\II{0}{+}(x) &\define\{(i,j)\in \SetN_\ConConUP\times\SetN_\VanConUP\mid\ConCon i(x)=0, \, \VanCon ij(x) >0\},\\
\II{0}{0}(x) &\define\{(i,j)\in \SetN_\ConConUP\times\SetN_\VanConUP\mid\ConCon i(x)=0, \, \VanCon ij(x) =0\},\\
\II{0}{-}(x) &\define\{(i,j)\in \SetN_\ConConUP\times\SetN_\VanConUP\mid\ConCon i(x)=0, \, \VanCon ij(x) <0\},\\
\I{0}(x) &\define\{(i,j)\in \II{0}{+}\cup\II{0}{0}\cup\II{0}{-}\},\\
\I{\ICon}(x) &\define\{k \in \SetN_\ICon \mid \ICon_k(x) = 0\}.
\end{split}
\end{align}
where we call the indices in $\II{0}{0}(x)$ and respective $x$ \emph{bi-active} (see Figure~\ref{fig:MPBVC_vertical}).
We also consider the indices for $\ConConUP$ in $\II \cdot \cdot (x)$ and set
\begin{equation*}
  \begin{aligned}
    \II{\cdot}{\cdot}^\ConConUP(x) \define \{i\in \SetN_\ConConUP \mid \exists j \in \SetN_\VanConUP \text{ with } (i,j) \in \II\cdot\cdot (x) \},\\
    \I{0}^\ConConUP(x) \define \{i\in \SetN_\ConConUP \mid \exists j \in \SetN_\VanConUP \text{ with } (i,j) \in \I{0} (x) \}.
  \end{aligned}
\end{equation*}
Furthermore for a given set $M\subset\SetN_\ConConUP\times\SetN_\VanConUP$ we write for short \[
  \ConCon{M^\ConConUP}(x) := \big(\ConCon{i}(x)\big)_{i\in M^\ConConUP}~\text{and}~
  \VanConUP_{M}(x) := \big(\VanCon{i}{j}(x)\big)_{(i,j)\in M}. 
\]
\begin{figure}[ht]
  \centering
  \includegraphics[width=.35\textwidth]{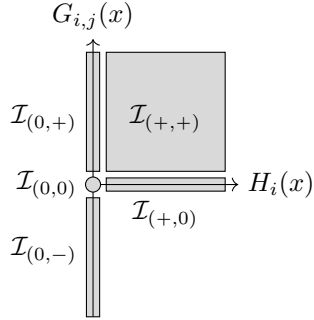}
  \caption{The feasible region of a controlling / vanishing constraint.}
  \label{fig:MPBVC_vertical}
\end{figure}

We let $\primalProd{x}{x'}$ denote the scalar product on the real
finite-dimensional Hilbert space $X$, inducing the norm
$\primalNorm{\cdot}$ for all $x, x' \in X$.  Analogously, we equip the
real finite-dimensional Hilbert spaces $Y$ and $\tilde{S}$ with scalar
products that induce their respective norms.

\section{Constraint Qualifications for Block Vanishing Problems}\label{sec:constraints_qualifications}
\begin{figure}[!t]
 	\centering
  \includegraphics{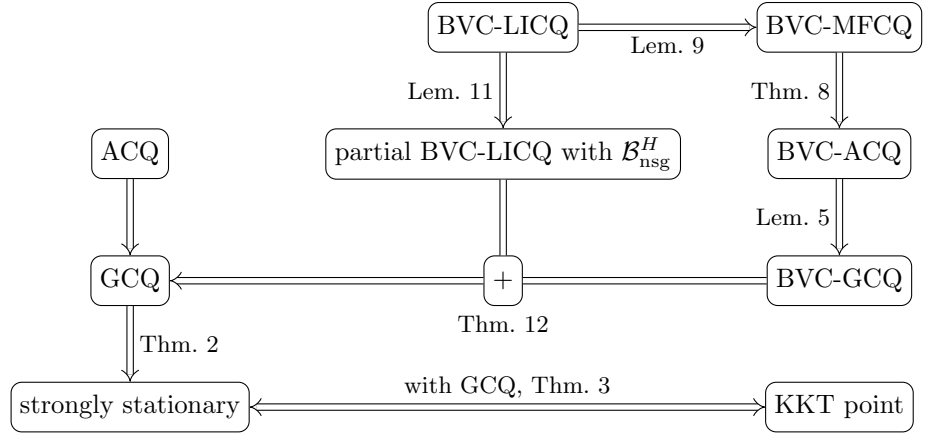}
 	\caption{Schematic overview of constraint qualifications and their implications.
  The strongest presented constraint qualification is the~\ref{word:BVC_LICQ} which leads both to strong stationarity, i.e. a KKT point for a feasible point.}
\label{vanishing: fig: CQ implikations}
\end{figure}

 In this section we derive a concept of stationarity in order to
 obtain practical first-order optimality conditions for MPBVCs.  The
 theoretical results are based
 on~\cite{diss_pal,mpbvc_pal,on_the_abadie,diss_hoheisel,mpvc}.  In
 particular we prove that if~\eqref{word:GCQ} holds for a local
 minimum of an MPBVC, the corresponding point is strongly stationary,
 \ie a KKT point.  Furthermore we show that~\eqref{word:BVC_LICQ}
 implies~\eqref{word:GCQ}. A schematic overview of introduced
 constraint qualifications and their corresponding relationships are
 given in Figure~\ref{vanishing: fig: CQ implikations}.

 In~\cite{mpbvc_pal}, a weaker formulation of~\ref{word:ACQ} was introduced, which is based on
 VC-ACQ from~\cite[Definition~3]{mpvc}. There are two common ways to
 define further constraint qualifications.  One is stricter and
 is based on~\cite[Section~4]{mpvc}. The other option is based
 on~\cite{diss_pal} and~\cite[Section~5]{mpvc} and leads
 to slightly different and also weaker version of KKT-like conditions
 (\textit{weak stationarity}), which gives weak necessary conditions for local optimality. 
 Our aim is to get KKT-points of \eqref{eq:mpbvc}. Note that we present here new formulations of constraint qualifications which differ
 from~\cite{diss_pal, mpbvc_pal}.
 
\subsection{Strong Stationarity under Guignard CQ}
We introduce a new cone, where we use more information of the problem structure. More precisely we include derivative
information in the linearized cone of all active constraints. The following results are based on~\cite[Section~3.3]{diss_pal}
and~\cite{on_the_abadie}.

\begin{lemma}
  \label{sec:linearized_cone_def}
  Let $x^\star\in\Real^n$ be a feasible point of~\eqref{eq:mpbvc}.
  Then the linearized cone of~\eqref{eq:mpbvc} at $x^\star$ is equal to
  \begin{align}
    \label{eq:mpbvc-L(x)}
    \Lag(x^\star)=\left\{
    d\in\Real^n \biggm| \begin{array}{rl}
      \DECon_k(x^\star)^Td~=0&\mathrm{for}\,\;k\in\SetN_\ECon,\\
      \DICon_l(x^\star)^Td~\geq0&\mathrm{for}\,\;l\in\I\ICon(x^\star),\\
      \DConCon i(x^\star)^Td\;~= 0&\mathrm{for}\,\;i\in\II{0}{-}^\ConConUP(x^\star),\\
      \DConCon i(x^\star)^Td\; ~\geq 0&\mathrm{for}\,\;i\in\II{0}{0}^\ConConUP(x^\star)\cup\II{0}{+}^\ConConUP(x^\star),\\
      \DVanCon ij(x^\star)^Td ~\geq 0&\mathrm{for}\,(i,j)\in\II{+}{0}(x^\star)
    \end{array}
    \right\}.
  \end{align}
\end{lemma}
\begin{proof} 
  The proof is based on the proof of~\cite[Lemma~4]{mpvc}.
  We note that the linearized cone of~\eqref{eq:mpbvc} is defined as
  \begin{align*}
    \left\{
    d\in\Real^n \biggm|
    \begin{array}{rl}
      \DECon_k(x^\star)^Td=0\,&\mathrm{for}\,k\,\in\SetN_\ECon,\\
      \DICon_l(x^\star)^Td\geq0\,&\mathrm{for}\,l\,\in\I\ICon(x^\star),\\
      \DConCon i(x^\star)^Td\,\geq 0\,&\mathrm{for}\,i\,\in\I0^\ConConUP(x^\star),\\
      \nabla \theta_{ij}(x^\star)^Td \geq 0\,&\mathrm{for}\,(i,j)\,\in\I0(x^\star)\cup\II{+}{0}(x^\star)
    \end{array}
    \right\},
  \end{align*}
  where $\theta_{ij}(x)\define\VanCon ij(x)\ConCon i(x)$.
  We distinguish the following cases with respect to the gradients $\nabla\theta_{ij}(x^\star)$:
\begin{itemize}
    \item For $(i,j) \in \II{0}{-}$ it holds that $\nabla\theta_{ij}(x^\star)^Td = \VanCon ij(x^\star)\DConCon i(x^\star)^Td \geq 0$
    which is equivalent to ${\DConCon i(x^\star)^Td \leq 0}$, since $\VanCon ij(x^\star)<0$. We obtain with $\DConCon i(x^\star)^Td\geq 0$ that $\DConCon i(x^\star)^Td=0$. 
    \item For $(i,j) \in \II{0}{0}$ it holds that $\nabla\theta_{ij}(x^\star)^Td = 0\geq 0$, \ie we can exclude this case, since it provides no information.
    \item For $(i,j) \in \II{0}{+}$ it holds that $\nabla\theta_{ij}(x^\star)^Td = \VanCon ij(x^\star)\DConCon i(x^\star)^Td\geq 0$ which is equivalent to  $\DConCon i(x^\star)^Td \geq 0$. This requirement is already included.
    \item For $(i,j) \in \II{+}{0}$ it holds that $\nabla\theta_{ij}(x^\star)^Td =\ConCon i(x^\star)\nabla \VanCon ij(x^\star)^Td \geq 0$. Since $\ConCon i(x^\star)\geq 0$, it follows that ${\DVanCon ij(x^\star)^Td\geq 0}$.
\end{itemize}
\vspace{-1.5em}\end{proof}
We are now able to define Abadie and Guignard constraint qualification.
\begin{definition}[Abadie and Guignard constraint qualification~\cite{abadie1966kuhn,guignard_stat_cons}]
  \label{vanishing: def: ACQ}
    Let $x$ be a feasible point of a standard NLP with feasible set $X$. We define the standard \textit{tangent cone} $\mathcal{T}(X,x)$ at $x \in X$
    and the \textit{polar cone} $M^{-}$ of a set $M$ as
    \begin{align}
        \mathcal{T}(X,x) &\define \bigg\{ d\in \Real^n \Bigm| \exists\{x^{(k)}\}\subset X,\{t^{(k)}\}\to0\,:\, x^{(k)}\to x\text{ and }\frac{x^{(k)}-x}{t^{(k)}}\to d \bigg\},\label{vanishing: T(x)}\\
        M^{-} &\define \{d \in \Real^{\tilde{n}} \mid \langle d, x \rangle \leq 0 \text{ for all } x \in M\}.
    \end{align}
    The point $x$ satisfies the \textit{Abadie constraint qualification} (\setword{ACQ}{word:ACQ}) if {$\Lag(x)=\mathcal{T}(X,x)$} and the \textit{Guignard constraint qualification} (\setword{GCQ}{word:GCQ}) if $\Lag(x)^{-}=\mathcal{T}(X,x)^{-}$.
\end{definition}
We define the  \textit{BVC Lagrange function} for~\ref{eq:mpbvc} by \begin{align}
\label{eq:mpbvc-Lagrangian}
    \Lag(x,\mu^\ECon,\mu^\ICon,\mu^\VanConUP,\mu^\ConConUP) &\define f(x)-\sum_{k\in\SetN_\ECon}\mu^\ECon_k\ECon_k(x) -\sum_{l\in\SetN_\ICon}\mu^\ICon_l\ICon_l(x)\\ &\enspace-\!\!\!\sum_{(i,j)\in\SetN_\ConConUP\times\SetN_\VanConUP} \!\!\!\mu^\VanConUP_{ij}\VanCon ij(x) -\sum_{i\in\SetN_\ConConUP}\mu^\ConConUP_i\ConCon i(x)\notag
\end{align}
with $x\in\Real^n$, the (classic) Lagrange multipliers $\mu^\ECon\in\Real^{n_\ECon}\mu^\ICon\in\Real^{n_\ICon}$, and \textit{BVC Lagrange multipliers} $\mu^\VanConUP\in\Real^{n_\VanConUP\cdot n_\ConConUP},\mu^\ConConUP\in\Real^{n_\ConConUP}$. 
Here, the product is split and every constraint is treated separately. The product structure is hidden in the BVC Lagrange multipliers and index sets.
The following theorem gives KKT like conditions for vanishing constraint problems and is analogous to~\cite[Theorem~3.16]{diss_pal} 
applied to~\ref{word:GCQ} instead of~\ref{word:ACQ}.\begin{restatable}{theorem}{vanishingKKT}\label{vanishing: thm: kkt}
    Let $x^\star\in\Real^n$ be a local minimum of problem~\eqref{eq:mpbvc} where~\ref{word:GCQ} holds at $x^\star$. Then there exist multipliers $\mu^\ECon\in\Real^{n_\ECon}$, $\mu^\ICon\in\Real^{n_\ICon}$, $\mu^\VanConUP\in\Real^{n_\VanConUP\cdot n_\ConConUP}$, and $\mu^\ConConUP\in\Real^{n_\ConConUP}$ such that \begin{align}\label{vanishing: kkt-lagrange-equation}
        \nabla_x\Lag(x,\mu^\ECon,\mu^\ICon,\mu^\VanConUP,\mu^\ConConUP)=0
    \end{align} 
    and \begin{align}\label{vanishing: kkt-lagrange-multiplier}
    	\ECon_k(x^\star)&=0& &\hspace{-9em}\text{for all }\;k\in \SetN_\ECon,\notag\\
    	\mu^\ICon_l &\geq0,\, \ICon_l(x^\star)\geq0, \, \mu^\ICon_l\ICon_l(x^\star)=0 & &\text{for all }\;l\in \SetN_\ICon,\notag \\
            \mu^\ConConUP_i &=0  & &\hspace{-9em}\text{for all }~i\in \II{+}{+}^\ConConUP(x^\star)\cup\II{+}{0}^\ConConUP(x^\star),\notag \\
            \mu^\ConConUP_i &\geq0 & &\hspace{-9em}\text{for all }~i\in\I0^\ConConUP \text{ if there exists no } j ~\text{with}~(i,j)\in \II{0}{-}(x^\star),  \notag \\
            \mu^\ConConUP_i &\text{ free} & &\hspace{-9em}\text{for all }~i\in\I0^\ConConUP \text{ if there exists } j~ \text{with}~ (i,j)\in \II{0}{-}(x^\star), \notag  \\
            \mu^\VanConUP_{ij} &\geq0 &&\hspace{-9em}\text{for all }(i,j)\in \II{+}{0}(x^\star), \notag\\
            \mu^\VanConUP_{ij} &=0 &&\hspace{-9em}\text{otherwise.}    
    \end{align}
  \end{restatable}
\begin{proof}
  The proof is standard and may be found in Appendix~\ref{sec:appendix}.
\end{proof}
\begin{definition}[Strongly stationary points for MPBVC]
  \label{vanishing: def: kkt}
  A point $x^\star\in\Real^n$ is a \emph{\setword{strongly stationary}{word:strong_stationary}} (or KKT) point if there exist multipliers $\mu^{*}$ satisfying
  equations~\eqref{vanishing: kkt-lagrange-equation} and~\eqref{vanishing: kkt-lagrange-multiplier}
  from Theorem~\ref{vanishing: thm: kkt}.
\end{definition}
The reason to stick to~\ref{word:GCQ} and strong stationarity is the
equivalence between strong stationarity of MPBVC and the KKT
conditions of an NLP:
\begin{restatable}{theorem}{vanishingEquiv}
  \label{vanishing: thm: equivalence of KKT points}
  If~\ref{word:GCQ} holds at a local minimum $x^\star$ of~\eqref{eq:mpbvc}, 
  then the classic KKT conditions and strong stationarity for MPBVC from Definition~\ref{vanishing: def: kkt} are equivalent.
\end{restatable}
\begin{proof}
  The proof may be found in Appendix~\ref{sec:appendix}.
\end{proof}
\subsection{Abadie and Guignard Type CQ for BVC}
Due to the fact that a single CC controls multiple VCs, the geometry
of the feasible set and its linearized and tangent cones is more
complicated than in the case of MPVCs. To capture the geometry
accurately, we introduce several cones based on a partition of
the bi-active controlling constraints into a subset of constraints
fixed to zero ($\ZeroIndices$) and another ($\PosIndices$), where
CCs can become positive. Formally, we let
$\PosIndices \cupdot \ZeroIndices = \II{0}{0}^\ConConUP(x^\star)$ be this
partition and set
${\PosIndicesVC\define\{(i,j)\in\II{0}{0}(x^\star)\,\mid\,i\in\PosIndices\}}$.
The partition gives rise to the following {$x^\star$-dependent} NLPs with feasible
regions contained in that of~\eqref{eq:mpbvc}:
\begin{mini}
	{x \in \Real^n}{f(x)}
	{\label{vanishing: NLP_beta}\tag{$\NLP(\PosIndices,\ZeroIndices,x^\star)$}}{}
	\addConstraint{\ECon_k(x)}{= 0,\quad k\,\in \SetN_\ECon}
	\addConstraint{\ICon_l(x)}{\geq 0,\quad l\,\in \SetN_\ICon}
	\addConstraint{\VanCon ij(x)}{\geq 0, \quad (i,j) \in \II{+}{+}(x^\star)\cup\II{0}{+}(x^\star)\cup\II{+}{0}(x^\star)}
	\addConstraint{\ConCon i(x)}{\geq 0, \quad i\, \in \II{+}{+}^\ConConUP(x^\star)\cup\II{0}{+}^\ConConUP(x^\star)\cup\II{+}{0}^\ConConUP(x^\star)}
	\addConstraint{\ConCon i(x)}{= 0, \quad i\, \in \II{0}{-}^\ConConUP(x^\star)}
	\addConstraint{\ConCon i(x)}{\geq 0, \quad i\, \in \PosIndices}
	\addConstraint{\VanCon ij(x)}{\geq 0, \quad (i,j) \in \PosIndicesVC}
	\addConstraint{\ConCon i(x)}{= 0, \quad i\, \in \ZeroIndices.}
\end{mini}
To define the linearized cone of \ref{vanishing: NLP_beta} we set
\begin{align}
  \label{vanishing: L(x)_super}
\Lag^{\super}(x^\star) \define \left\{
d\in\Real^n \biggm| \begin{array}{rl}
	\DECon_k(x^\star)^Td=0&\text{for}~k\,\in\SetN_\ECon,\\
	\DICon_l(x^\star)^Td\geq0&\text{for}~l\,\in\I\ICon(x^\star),\\
	\DConCon i(x^\star)^Td = 0&\text{for}~i\,\in\II{0}{-}^\ConConUP(x^\star),\\
	\DConCon i(x^\star)^Td \geq 0&\text{for}~i\,\in\II{0}{+}^\ConConUP(x^\star),\\
	\DVanCon ij(x^\star)^Td \geq 0&\text{for}~(i,j)\,\in\II{+}{0}(x^\star)\\
\end{array}\right\}.
\end{align}
The linearized cone of \ref{vanishing: NLP_beta} is then given by
\begin{align}
  \label{vanishing: L(x)_beta}
\Lag^{\NLP(\PosIndices,\ZeroIndices,x^\star)}(x^\star)& \define \Lag^{\super}(x^\star)\\
&\quad\cap\left\{
d\in\Real^n \biggm| \begin{array}{rl}
	\DConCon i(x^\star)^Td \geq 0&\text{for}~i\,\in\PosIndices\\
	\nabla \VanCon ij(x^\star)^Td \geq 0&\text{for}~(i,j)\,\in\PosIndicesVC\\
	\DConCon i(x^\star)^Td = 0&\text{for}~i\,\in\ZeroIndices\\
\end{array}\right\}.\notag
\end{align}
The \emph{modified linearized cone} $\Lag^\mathrm{mod}(x^\star) \define \Lag^{\NLP(\II{0}{0}(x^\star),\emptyset,x^\star)}(x^\star)$,
which appears for instance in~\cite{mpbvc_pal,diss_pal}, is not part of our examination,
as it only leads to a weak stationarity concept for MPBVC.
The cones given in~\eqref{vanishing: L(x)_beta} have a connection to the following cone:
\begin{definition}[BVC-linearized cone]\label{vanishing: def: BVC-linearized cone}
Let $x^\star$ be a feasible point of~\eqref{eq:mpbvc}.
The \emph{BVC-linearized cone} is defined as \begin{align}\label{eq:mpbvc-L(x)beta}
        \Lag^\mathrm{BVC}(x^\star)&\define\Lag^{\super}(x^\star)\\
&\hspace{-4.3em}\cap\left\{
d\in\Real^n \biggm| \begin{array}{rl}
	\ConCon i(x^\star)^Td \geq 0,&\text{for}~i\,\in\II{0}{0}^\ConConUP(x^\star)\\
(\DConCon i(x^\star)^Td)\cdot(\nabla \VanCon ij(x^\star)^Td)\geq 0,&\text{for}~(i,j)\,\in\II{0}{0}(x^\star)\\
\end{array}\right\}.\notag\end{align} 
\end{definition} 
\begin{lemma}
  \label{vanishing: lem. T(x) and L(x) beta}
	Let $x^\star$ be a feasible point of a MPBVC. Then it holds
  \begin{enumalpha}
  \item \label{vanishing: lem. T(x) and L(x) beta: caseA} $\displaystyle{\mathcal{T}(X,x^\star) = \bigcup_{\PosIndices \cupdot \ZeroIndices = \II{0}{0}^\ConConUP(x^\star)}\mathcal{T}^{\NLP(\PosIndices,\ZeroIndices,x^\star)}(x^\star)}$
  \item \label{vanishing: lem. T(x) and L(x) beta: caseB} $\displaystyle{\Lag^{\mathrm{BVC}}(x^\star) = \bigcup_{\PosIndices \cupdot \ZeroIndices = \II{0}{0}^\ConConUP(x^\star)}\Lag^{\NLP(\PosIndices,\ZeroIndices,x^\star)}(x^\star)}$
  \item \label{vanishing: lem. T(x) and L(x) beta: caseC} $\mathcal{T}(X,x^\star) \subseteq\Lag^{\mathrm{BVC}}(x^\star)$
  \end{enumalpha}
\end{lemma}
\begin{proof}
  This proof is analogously to the proof of~\cite[Lemma~2.4]{on_the_abadie}. For completeness we state the proof for MPBVC.
  We start with~\ref{vanishing: lem. T(x) and L(x) beta: caseA}.
   Assume $d \in \mathcal{T}(X,x^\star)$, then there exists a sequence
    $\{x^{(k)}\}$ of points feasible for MPBVC such that $x^{(k)} \to x^\star$
    for $k \to \infty$ which together with a sequence $\{t^{(k)}\}$ of
    step lengths converges to $d$. Our goal is to identify a subsequence
    which is contained entirely in the feasible set of
    $\NLP(\PosIndices,\ZeroIndices,x^\star)$ for a suitably chosen partition.
    Due to the feasibility of the iterates $x^{k}$ it holds that
    $\ConCon i(x^{(k)}) \geq 0$. From the continuity of the
    constraint functions it follows that starting from a sufficiently
    large index $K$ it holds that
    $\ConCon i(x^{(k)})=0$ for $i \in \II{0}{-}^\ConConUP(x^\star)$ and
    $\VanCon ij (x^{(k)}) > 0$ for
    $(i,j) \in \II{+}{+}(x^\star) \cup \II{0}{+}(x^\star)$
    for all $k \geq K$ (we can assume that $K = 1$).
    To find a suitable partition of the bi-active indices at $x^\star$ we set  for all $k\in\Nat$
    \begin{align*}
      \mathcal{B}^\ConConUP_{0,k}& \define \{i\,\mid\,\text{there exists } (i,j) \in \II{0}{0}(x^\star)~\text{with}~ (i, j) \in \II{0}{-}(x^{k})\}, \\
      \mathcal{B}^\ConConUP_{+,k}& \define \II{0}{0}^\ConConUP(x^\star)  \setminus\mathcal{B}^\ConConUP_{0,k}.
    \end{align*}
   We know that the number of partitions of $\II{0}{0}(x^\star)$ is finite. We can therefore choose an infinite
    subset $\mathcal{K}\subset\Nat$, such that $\widebar{\PosIndices} \define\mathcal{B}^\ConConUP_{+,k}$ and $\widebar{\mathcal{B}^\ConConUP_{0}}\define\mathcal{B}^\ConConUP_{0,k}$
    for all $k \in \mathcal{K}$. We see that $x^{(k)}$ is feasible for all $k\in\mathcal{K}$ for $\NLP(\widebar\PosIndices,\widebar\ZeroIndices)$.
    Therefore, the subsequence based on $\mathcal{K}$ shows that $d$ is contained in $\mathcal{T}^{\NLP(\widebar\PosIndices,\widebar\ZeroIndices)}(x^\star)$.

    Conversely, for a value $d$ in the union of tangential cones there exists a given partition $(\PosIndices,\ZeroIndices)$
    such that $d\in\mathcal{T}^{\NLP(\PosIndices,\ZeroIndices,x^\star)}(x^\star)$ with corresponding sequences $\{x^{(k)}\},\{t^{(k)}\}$. Since the feasible set
    of $\NLP(\PosIndices,\ZeroIndices,x^\star)$ is contained within that of~\eqref{eq:mpbvc}, the 
    $x^{(k)}$ are all feasible for \eqref{eq:mpbvc}, showing that $d \in \mathcal{T}(X,x^\star)$.

  We carry on with~\ref{vanishing: lem. T(x) and L(x) beta: caseB}.
    Assume $d\in\Lag^\mathrm{BVC}(x^\star)$ and take the partition
    \begin{align*}
      \ZeroIndices & \define \{i\,\mid\,\text{there exists }(i,j)\in\II{0}{0}(x^\star)~\st~ \DVanCon ij(x^\star)^Td<0\},\\
      \PosIndices & \define \II{0}{0}^\ConConUP(x^\star) \setminus\ZeroIndices.
    \end{align*} From the product $(\DConCon i(x^\star)^Td)\cdot(\nabla \VanCon ij(x^\star)^Td)\geq0$ and $\DConCon i(x^\star)^Td\geq0$ for all $(i,j)\in\II{0}{0}(x^\star)$
    we obtain with the choice of $\ZeroIndices$ that $\DConCon i(x^\star)^Td=0$ for all $i\in\ZeroIndices$. For $\PosIndices$ we obtain $\DConCon i(x^\star)^Td\geq0$
    and $\nabla \VanCon ij(x^\star)^Td\geq0$ for all $(i,j)\in\PosIndicesVC$.
 The converse direction is clear by definitions of the respective cones.
 For~\ref{vanishing: lem. T(x) and L(x) beta: caseC} we know that $\mathcal{T}^{\NLP(\PosIndices,\ZeroIndices,x^\star)}(x^\star)\subseteq\Lag^{\NLP(\PosIndices,\ZeroIndices,x^\star)}(x^\star)$
  for all partitions $\PosIndices \cupdot \ZeroIndices = \II{0}{0}^\ConConUP(x^\star)$. If follows
  \begin{align*}
   	\mathcal{T}(X,x^\star) \overset{\ref{vanishing: lem. T(x) and L(x) beta: caseA}}{=} & \bigcup_{\PosIndices \cupdot \ZeroIndices = \II{0}{0}^\ConConUP(x^\star)} \mathcal{T}^{\NLP(\PosIndices,\ZeroIndices,x^\star)}(x^\star)
    \\ \subseteq &
    \bigcup_{\PosIndices \cupdot \ZeroIndices = \II{0}{0}^\ConConUP(x^\star)} \Lag^{\NLP(\PosIndices,\ZeroIndices,x^\star)}(x^\star) \overset{(b)}{=}\Lag^{\mathrm{BVC}}(x^\star).
  \end{align*}
  This is our assertion.
\end{proof}
Note that $\Lag^\mathrm{BVC}(x^\star)\subset\Lag(x^\star)$ and equality holds when $\II{0}{0}=\emptyset$ and shows the importance of the bi-active indices.
\begin{definition}\label{vanishing: def: BVC-ACQ}
	Let $x^\star$ be a feasible point of~\eqref{eq:mpbvc}. We say the
  \begin{enumalpha}
  \item \emph{BVC-Abadie constraint qualification} (\setword{BVC-ACQ}{word:BVC_ACQ}) holds at  $x^\star$ if
    \begin{equation*}
      \Lag^\mathrm{BVC}(x^\star)=\mathcal{T}(X,x^\star).
    \end{equation*}
  \item \emph{BVC-Guignard constraint qualification} (\setword{BVC-GCQ}{word:BVC_GCQ}) holds at  $x^\star$ if
    \begin{equation*}
      \Lag^\mathrm{BVC}(x^\star)^{-}=\mathcal{T}(X,x^\star)^{-}.
    \end{equation*}
    \end{enumalpha}
\end{definition}
\begin{lemma}\label{vanishing: lem: BVC-ACQ implies BVC-GCQ}
	Let $x^\star$ be a feasible point. The following statements hold
  \begin{enumalpha}
  \item \ref{word:BVC_ACQ} implies \ref{word:BVC_GCQ}.
  \item
    If~\ref{word:ACQ} holds for~\eqref{vanishing: NLP_beta} for every partition $\PosIndices \cupdot \ZeroIndices = \II{0}{0}^\ConConUP(x^\star)$, then~\ref{word:BVC_ACQ} holds.
  \item
    If~\ref{word:GCQ} holds for~\eqref{vanishing: NLP_beta} for every partition $\PosIndices \cupdot \ZeroIndices = \II{0}{0}^\ConConUP(x^\star)$, then~\ref{word:BVC_GCQ} holds.
  \end{enumalpha}
\end{lemma}
\begin{lemma}\label{vanishing: cor: BVC-ACQ = ACQ}
  If $\II{0}{0}(x^\star)=\emptyset$ holds for a feasible point $x^\star$ of~\eqref{eq:mpbvc}, then~\ref{word:BVC_ACQ} is equivalent to~\ref{word:ACQ} and~\ref{word:BVC_GCQ} is equivalent to~\ref{word:GCQ}.
\end{lemma}
\begin{proof}
		Let $x^\star$ be a feasible point with $\II{0}{0}(x^\star) =\emptyset$. We obtain the inclusions\begin{align*}
			&\Lag(x^\star) = \Lag^\mathrm{BVC}(x^\star)\subseteq \mathcal{T}(X,x^\star)\subseteq\Lag(x^\star) \text{ for ACQ},\\
			&\Lag(x^\star)^{-} = \Lag^\mathrm{BVC}(x^\star)^{-}\supseteq \mathcal{T}(X,x^\star)^{-}\supseteq\Lag(x^\star)^{-} \text{ for GCQ},
		\end{align*}which leads to equalities.
  \end{proof}
Examples where~\ref{word:ACQ} is violated with $\II{0}{0}(x^\star)\neq\emptyset$ can be found in~\cite{mpvc}.
In general it is shown in~\cite[Section~5]{mpvc} that~\ref{word:BVC_ACQ} is weaker then~\ref{word:ACQ} already for MPVCs.\\
\subsection{Mangasarian-Fromovitz and Linear Independence Type CQ for BVC}
If we would like to apply the classic \emph{Mangasarian-Fromovitz Constraint Qualification} (MFCQ) \cite{mfcq}
we face the problem that due to the product structure we cannot guarantee $\nabla\big(\ConCon i(x)\VanCon ij(x)\big)^Td>0$ 
for all active indices. First of all, for the indices contained in $\II{0}{-}(x)$ 
we know that the $\ConCon{i}$ are active and we need to treat them as equality constraint. 
Secondly, the index set $\II{0}{0}(x)$ may have indices $(i,j)$ with $\VanCon ij(x)^Td<0$ for a given direction $d$. 
So we need to distinguish in every partition of $\PosIndices\cupdot\ZeroIndices=\II{0}{0}^\ConConUP(x)$ between indices 
where we could move into the interior of feasible set ($\PosIndices)$ and 
indices where the vanishing constraint may vanish ($\ZeroIndices$). 
\begin{definition}[BVC-MFCQ]\label{vanishing: def: BVC-MFCQ}
  Let $x^\star\in\Real^n$ be a feasible point of~\eqref{eq:mpbvc}. 
  We say that \emph{BVC-Mangasarian-Fromovitz Constraint Qualification}
  (\setword{BVC-MFCQ}{word:BVC_MFCQ})
  is satisfied at $x^\star$ if for every partition $\PosIndices\cupdot\ZeroIndices=\II{0}{0}^\ConConUP(x^\star)$ 
  MFCQ holds for \ref{vanishing: NLP_beta}. Precisely, for every partition $\PosIndices\cupdot\ZeroIndices=\II{0}{0}^\ConConUP(x^\star)$  the gradients 
  \begin{align*}
    \DECon_k(x^\star)\text{ for }k\in\SetN_\ECon,\enspace\DConCon i(x^\star) \text{ for }i \in \II{0}{-}^\ConConUP(x^\star)\cup\ZeroIndices
  \end{align*} are linearly independent and a vector $d\in\Real^n$ exists such that \begin{align*}
    \DECon_k(x^\star)^Td = 0, &~ k\, \in\SetN_\ECon,\\
    \DICon_l(x^\star)^Td > 0, & ~l\, \in\I\ICon(x^\star),\\
    \DConCon i(x^\star)^Td = 0,& ~ i\, \in\II{0}{-}^\ConConUP(x^\star)\cup\ZeroIndices,\\
    \DConCon i(x^\star)^Td > 0,&  ~i\,\in\I{0}^\ConConUP(x^\star) \setminus (\II{0}{-}^\ConConUP(x^\star)\cup\ZeroIndices),\\
    \DVanCon{i}{j}(x^\star)^Td > 0,&~  (i,j)\,\in \II{+}{0}(x^\star)\cup\PosIndicesVC.
  \end{align*}
\end{definition}
Note that this definition differs from the well-known standard MFCQ~\cite{mfcq}
for general NLPs in the sense that some of the inequalities are handled as equality constraints.
This is due to the more complicated structure of~\eqref{eq:mpbvc}. 
Indeed, the \ref{word:BVC_MFCQ} is a collection of the MFCQ conditions for all possible \ref{vanishing: NLP_beta}.
In order for a (nonzero) value of $d$
to exist, it has to hold that $\#\SetN_\ECon + \#\II{0}{-}^\ConConUP(x^\star) + \#\II{0}{0}^\ConConUP(x^\star) < n$ to cover all partitions of $\II{0}{0}^\ConConUP(x^\star)$.
\begin{lemma}\label{vanishing: lem: BVC-MFCQ}
    Let $x^\star$ be feasible point of~\eqref{eq:mpbvc} satisfying~\ref{word:BVC_MFCQ}.
    Then there exist a $\epsilon>0,\,d\in\Real^n$, and a continuously differentiable function $x:(-\epsilon,\epsilon)\to\Real^n$ such that $x(0)=x^\star$, $x'(0)=d$, and $x(t)$ is feasible for all $t\in[0,\epsilon)$. 
\end{lemma}
\begin{proof}
    This proof is based on the Implicit Function Theorem (IFT) and modified from~\cite{mpvc,mpbvc_pal}.
    In the first step we use the IFT to show the existence of a continuously differentiable function.
    We then show that $x(t)$ is feasible for all $t\in[0,\epsilon)$.
    Let a partition $\PosIndices\cupdot\ZeroIndices=\II{0}{0}^\ConConUP(x^\star)$ be given and fixed.\\

    \textsc{Step 1.} We define the mapping $z:\Real^n\to\Real^q$ with \newline $q\define\#\SetN_\ECon+
    \#\II{0}{-}^\ConConUP(x^\star)+ \#\ZeroIndices$ and
    \begin{align*}
        z(x)\define\begin{pmatrix}
          \ECon(x) \\
          \ConCon{\II{0}{-}^\ConConUP(x^\star)}(x) \\
          \ConCon{\ZeroIndices}(x)\\
        \end{pmatrix}.
    \end{align*}We denote with $z_j$ the $j$-th component of the mapping $z$. 
    Let further $t \in \Real, y\in\Real^q$, $d\in\Real^n$ from~\ref{word:BVC_MFCQ} and 
    define $\hat{h}: \Real^{q} \times \Real \to \Real^q$ with
    \begin{align*}
      \hat{h}_j(y,t)\define z_j(x^\star + td + J_z(x^\star)^Ty) \text{ for all }j=1,\dots,q,
    \end{align*}where $J_z(x^\star)$ denotes the Jacobian of $z$ evaluated at $x^\star$.
    Then the (nonlinear) system $\hat{h}(y,t)=0$ has a solution $(y^\star,t^\star)\define(0,0)$ and the
    partial Jacobian $\frac{\partial}{\partial y}\hat{h}(0,0)=J_z(x^\star)J_z(x^\star)^T$ has full 
    rank since $x^\star$ satisfies~\ref{word:BVC_MFCQ}.
    It follows from the IFT that there exists an $\epsilon>0$ and a continuously differentiable function
    $y:(-\epsilon,\epsilon)\to\Real^q$ such 
    that $y(0)=0$ and $\hat{h}(y(t),t)=0$ for all $t\in(-\epsilon,\epsilon)$. 
    Its derivative is given by \begin{align*}
        \frac{d}{dt} y(t) = -\left(\frac{\partial}{\partial y}\hat{h}(y(t),t)\right)^{-1}\frac{\partial}{\partial t}\hat{h}(y(t),t)\;\text{ for all }t\in(-\epsilon,\epsilon).
    \end{align*} Moreover, we obtain with the chain rule
    \begin{align*}
    \frac{d}{dt} y(0) = -\left(\frac{\partial}{\partial y}\hat{h}(y(0),0)\right)^{-1}\frac{\partial}{\partial t}\hat{h}(y(0),0) 
    = -\left(\frac{\partial}{\partial y}\hat{h}(y(0),0)\right)^{-1} J_z(x^\star) d =0,
    \end{align*}
    since $J_z(x^\star) d=0$ from~\ref{word:BVC_MFCQ}. 
    We set \begin{align*}
        x(t)\define x^\star + td + J_z(x^\star)^Ty(t),
    \end{align*} which is by definition continuously differentiable on the interval $(-\epsilon,\epsilon)$. 
    Furthermore it holds that
    \begin{align*}
        x(0) &= x^\star + 0\cdot d + J_z(x^\star)^Ty(0) = x^\star + J_z(x^\star)^T\cdot 0=x^\star \text{, and}\\
         \frac{d}{dt}x(0)&= d+ J_z(x^\star)^T  \frac{d}{dt}y(0) = d + J_z(x^\star)^T\cdot 0 = d.
    \end{align*}
    
    \textsc{Step 2.} We note that $0 = \hat{h}(y(t),t) = z(x(t))$ and by definition of $z$ it follows that \begin{equation}
      \label{vanishing:\ConCon i(x(t))=0}
      \begin{aligned}        
        \ECon_k(x(t)) &=0 \text{ for all } k\in\SetN_\ECon\text{ and } t\in (-\epsilon,\epsilon) \text{, and}\\
       \ConCon i(x(t)) &=0 \text{ for all }i\in\II{0}{-}^\ConConUP(x^\star)\cup\ZeroIndices \text{ and } t\in (-\epsilon,\epsilon).
      \end{aligned}
    \end{equation} Furthermore, by continuity it holds that $\ConCon i(x(t))\geq0$ 
    for all $i \in \II{+}{0}^\ConConUP(x^\star) \cup \II{+}{+}^\ConConUP(x^\star)$ 
    and sufficiently small $t>0$.

    We define $\phi(t)\define\ConCon i(x(t))$ for the indices ${i \in \I{0}^\ConConUP(x^\star) \setminus (\II{0}{-}^\ConConUP(x^\star)\cup\ZeroIndices)}$. 
    It follows that $\frac{d}{dt} \phi(t) = \nabla_x\ConCon i(x(t))^T \frac{d}{dt} x(t)$ and so 
    $\frac{d}{dt} \phi(0) = \nabla_x\ConCon i(x^\star)^Td>0$ by~\ref{word:BVC_MFCQ}, i.e.
     $\ConCon i(x(t))\geq0$ for $i\in\SetN_\ConConUP$ and sufficiently small $t\geq0$. 
     Analogously we obtain $\ICon_l(x(t))\geq0$ for all $l\in\SetN_\ICon$ and $\VanCon{i}{j}(x(t))\geq0$ for $(i,j)\in\II{+}{0}(x^\star)\cup\PosIndicesVC$.
     It remains to show that $\varphi(x(t))\define\ConCon i(x(t))\VanCon ij(x(t))\geq 0$ for $(i,j)\in\SetN_\ConConUP\times\SetN_\VanConUP$. 
     For $\II{0}{-}(x^\star)\cup\ZeroIndices$ it is satisfied by~\eqref{vanishing:\ConCon i(x(t))=0}. Let in the following $t>0$ be sufficiently small.
     By continuity, $\varphi(x(t))\geq0$ for $(i,j)\in\II{+}{+}(x^\star)$. 
     Let $(i,j)\in\II{0}{+}(x^\star)\cup\II{+}{0}(x^\star)$ and $i\not\in\II{0}{-}^\ConConUP(x^\star)\cup\ZeroIndices$. 
     We obtain \begin{align*}
        \frac{d}{dt}\varphi(t) = \VanCon ij(x(t))\DConCon i(x(t))^Tx'(t)+\ConCon i(x(t))\DVanCon ij(x(t))^Tx'(t) 
    \end{align*} and so \begin{align*}
        \frac{d}{dt}\varphi(0) = \VanCon ij(x^\star)\DConCon i(x^\star)^Td+\ConCon i(x^\star)\DVanCon ij(x^\star)^Td > 0,
    \end{align*}since~\ref{word:BVC_MFCQ} holds. From $\varphi(0) = 0$ we obtain $\varphi(t)>0$ for sufficiently small $t>0$.
    The case $(i,j) \in \PosIndicesVC$ remains. We know $\frac{d}{dt}\varphi(0)=0$, 
    but $\ConCon{i}(x(t))\geq0$ and $\VanCon{i}{j}(x(t))\geq0$ are feasible by~\ref{word:BVC_MFCQ}.
    It follows, that $\varphi(t)\geq0$.
    We conclude that $\ConCon i(x(t))\VanCon ij(x(t))\geq 0$ for sufficiently small $t\geq0$.
\end{proof}
The following theorem is adapted from~\cite[Theorem~2]{mpvc}.
\begin{theorem}\label{vanishing: thm: BVC-MFCQ implies BVC-ACQ}
 Let $x^\star$ be a feasible point of~\eqref{eq:mpbvc} where~\ref{word:BVC_MFCQ} holds.
 Then~\ref{word:BVC_ACQ} holds at $x^\star$.
\end{theorem}
\begin{proof} We prove that $\Lag^\mathrm{BVC}(x^\star)\subseteq\mathcal{T}(X,x^\star)$. 
  Let $d\in\Lag^\mathrm{BVC}(x^\star)$. From Lemma~\ref{vanishing: lem. T(x) and L(x) beta}~\ref{vanishing: lem. T(x) and L(x) beta: caseB}
  we can find a partition $\PosIndices \cupdot \ZeroIndices = \II{0}{0}^\ConConUP(x^\star)$
  such that $d\in\Lag^{\NLP(\PosIndices,\ZeroIndices,x^\star)}$, \ie for the vanishing and controlling constraints it holds\begin{align*}
            \DConCon i(x^\star)^Td &= 0 &&\text{for all }i\, \in \II{0}{-}^\ConConUP(x^\star),\\
            \DConCon i(x^\star)^Td &\geq 0 &&\text{for all }i\, \in \II{0}{+}^\ConConUP(x^\star), \\
            \nabla \VanCon ij(x^\star)^Td &\geq 0 &&\text{for all }(i,j) \in \II{+}{0}(x^\star),\\
            \DConCon i(x^\star)^Td &\geq 0&&\text{for all }i\,\in\PosIndices,\\
            \nabla \VanCon ij(x^\star)^Td &\geq 0&&\text{for all }(i,j)\in\PosIndicesVC,\\
            \DConCon i(x^\star)^Td &= 0&&\text{for all }i\,\in\ZeroIndices.
        \end{align*} Let $\hat{d}\in\Real^n$ such that it satisfies~\ref{word:BVC_MFCQ} with the same partition
        and define $d_{\delta}\define d+\delta\hat{d}$. It follows that
        \begin{subequations}\label{vanishing: MFCQ->ACQ eq}
        \begin{align*}
        &\DConCon i(x^\star)^Td_{\delta}= \DConCon i(x^\star)^Td + \delta \DConCon i(x^\star)^T\hat{d} = 0\\
        &\hspace{2.5em}\text{for all } i\in \II{0}{-}^\ConConUP(x^\star)\cup\ZeroIndices, \\
        &\DConCon i(x^\star)^Td_{\delta} = \DConCon i(x^\star)^Td + \delta \DConCon i(x^\star)^T\hat{d} > 0\\
        &\hspace{2.5em} \text{for all } i\in \II{0}{+}^\ConConUP(x^\star)\cup\PosIndices. 
        \shortintertext{We move on to the active vanishing constraints}
        &\nabla \VanCon ij(x^\star)^Td_{\delta}
        = \nabla \VanCon ij(x^\star)^Td + \delta \nabla \VanCon ij(x^\star)^T\hat{d} > 0\\
        &\hspace{2.5em} \text{for all } (i,j)\in \II{+}{0}(x^\star)\cup\PosIndicesVC 
        \end{align*}
        \end{subequations}
        for $\delta>0$. 
        It follows that $d_{\delta}$ satisfies the~\ref{word:BVC_MFCQ}
        conditions for sufficiently small $\delta$. 
        Let $\delta$ be fixed. Since~\ref{word:BVC_MFCQ} holds for $x^\star$ and $d_{\delta}$ we obtain with Lemma~\ref{vanishing: lem: BVC-MFCQ}
        that there exists an $\epsilon^{(\delta)}>0$ and a smooth curve $x_{\delta}:(-\epsilon,\epsilon)\to\Real^n$
        such that $x_{\delta}(0)=x^\star$, $\frac{d}{dt} x_{\delta}(0)=d_{\delta}$
        and $x_{\delta}(t)$ is feasible for all $t\geq0$ sufficiently small. We take a sequence $\{t^{(k)}\}$ with $t^{(k)}$ sufficiently small,
        $t^{(k)}\to0$ for $k\to\infty$ and define $x^{(k)}\define x_{\delta}(t^{(k)})$.
        Then $x^{(k)}$ is feasible and $x^{(k)}\to x^\star$ for $k\to\infty$. Hence, \begin{align*}
            d_{\delta} = \frac{d}{dt} x_{\delta}(0) = \lim_{k\to\infty}\frac{x_{\delta}(t^{(k)})-x_{\delta}(0)}{t^{(k)}} =\lim_{k\to\infty}\frac{x^{(k)}-x^\star}{t^{(k)}},
        \end{align*} \ie $d_{\delta}\in\mathcal{T}(X,x^\star)$ for every sufficiently small $\delta>0$.
        Since the cone $\mathcal{T}(X,x^\star)$ is closed we can let $\delta \to 0$ and get $ d_{\delta}\to d\in\mathcal{T}(X,x^\star)$.
\end{proof}
We now state the adapted linear independence constraint qualification, introduced in \cite{mpbvc_pal,diss_pal}. 
As in classic nonlinear programming, this constraint qualification is a strong regularity condition and implies weaker
constraint qualifications.
\begin{definition}[BVC-LICQ]\label{vanishing: def: BVC-LICQ}
Let $x^\star$ be a feasible point of~\eqref{eq:mpbvc}. We say $x^\star$ satisfies the \emph{linear independence constraint qualification} (\setword{BVC-LICQ}{word:BVC_LICQ}) if the gradients \begin{align*}
	\DECon_k(x^\star)& \text{ for all }k \in \SetN_\ECon,\\
    \DICon_l(x^\star)& \text{ for all }l \in \I\ICon(x^\star),\\
    \DConCon i(x^\star)& \text{ for all }i \in \I0^\ConConUP(x^\star), \\
    \DVanCon ij(x^\star)& \text{ for all }(i,j) \in \II{0}{0}(x^\star)\cup\II{+}{0}(x^\star)
\end{align*} are linearly independent.
\end{definition}
The following lemma gives the relation between \ref{word:BVC_LICQ} and \ref{word:BVC_MFCQ}.
\begin{lemma}\label{vanishing: lem: BVC-LICQ implies BVC-MFCQ}
    Let $x^\star$ a feasible point of~\eqref{eq:mpbvc}. If~\ref{word:BVC_LICQ} holds at $x^\star$, then so does~\ref{word:BVC_MFCQ}.
\end{lemma}
\begin{proof}
	Since~\ref{word:BVC_LICQ} holds, it follows that LICQ holds
  for~\ref{vanishing: NLP_beta} for every partition ${\PosIndices \cupdot \ZeroIndices = \II{0}{0}^\ConConUP(x^\star)}$.
  Consequently, MFCQ holds for all of these NLPs as well, which
  by its definition implies~\ref{word:BVC_MFCQ}.
\end{proof}
As a consequence, it follows that \ref{word:BVC_LICQ} implies \ref{word:BVC_GCQ}.
\begin{theorem}\label{vanishing: thm: BVC-LICQ implies ACQ}
 Let $x^\star$ be a local minimum of~\eqref{eq:mpbvc} such that~\ref{word:BVC_LICQ} holds at $x^\star$. Then~\ref{word:BVC_ACQ} and~\ref{word:BVC_GCQ} hold at $x^\star$ and the Lagrange multipliers for~\eqref{eq:mpbvc-Lagrangian} are unique.
\end{theorem}
\begin{proof}
    Clear by definitions of~\ref{word:BVC_ACQ} and~\ref{word:BVC_GCQ}, Lemma~\ref{vanishing: lem: BVC-LICQ implies BVC-MFCQ}, and Theorem~\ref{vanishing: thm: BVC-MFCQ implies BVC-ACQ}.
    The linear independence of the gradients results in unique Lagrange multipliers.
\end{proof}
\subsection{Deriving Guignard CQ from BVC Constraint Qualifications}
We stated constraint qualifications implying~\ref{word:BVC_GCQ} so far (see Figure~\ref{vanishing: fig: CQ implikations}). 
What we actually need is \ref{word:GCQ} to get the strong stationarity conditions representing first order optimality conditions.  
The next definitions and results are motivated by~\cite{on_the_abadie} to have an implication from~\ref{word:BVC_LICQ} together with~\ref{word:BVC_GCQ} to~\ref{word:GCQ}.
\begin{definition}\label{vanishing: def: nonsingular ineq}
	Consider the linear system
  \begin{align}
    \label{vanishing: nonsingular system}
		Ad \geq r \text{ and } B d = q, 
	\end{align}
  where $A,B$ are matrices, and $d,r,q$ are vectors of appropriate dimensions.
  A single inequality $a_id \geq r_i$  with a row $a_i$ of the matrix $A$ is called \emph{nonsingular}
  if there exists a solution of system~\eqref{vanishing: nonsingular system} such that $a_id>r$ holds.
  Otherwise we call this inequality \emph{singular}, \ie $a_id=r_i$ is the only occurring solution.
\end{definition}
If we examine the system defining $\Lag(x^\star)$ in~\eqref{eq:mpbvc-L(x)}, which we can rewrite as
\begin{align*}
\begin{pmatrix}
	\DICon_{\I\ICon(x^\star)}(x^\star)^T \\
	\DConCon {\II{0}{+}^\ConConUP(x^\star)\cup\II{0}{0}^\ConConUP(x^\star)}(x^\star)^T \\
	\nabla \VanConUP_{\II{+}{0}(x^\star)} (x^\star)^T \\
\end{pmatrix}\cdot d &\geq 0 \text{, and}\\ \begin{pmatrix}
\DECon_{\SetN_\ECon}(x^\star)^T \\
\DConCon {\II{0}{-}^\ConConUP(x^\star)}(x^\star)^T \\
\end{pmatrix}\cdot d &= 0,
\end{align*} then we can define a partition of $\II{0}{0}^\ConConUP(x^\star)$ by setting
\begin{equation}\label{vanishing: def: betaH}
  \begin{aligned}
    \BetaIndices &\define\{i\,\mid\, \DConCon{i}(x^\star)^Td\geq0 \text{ is nonsingular for } \II{0}{0}^\ConConUP(x^\star) \},\\
    \ZeroBetaIndices &\define\{i\,\mid\, \DConCon{i}(x^\star)^Td\geq0 \text{ is singular for } i \in\II{0}{0}(x^\star)\}
  \end{aligned}
\end{equation} with respective $\BetaIndicesVC$ and $\ZeroBetaIndicesVC$. 
We can set $\PosIndices \define \BetaIndices$ and $\ZeroIndices \define \ZeroBetaIndices$.
Furthermore we define the following cone 
\begin{equation}\label{vanishing: LLxstar with betaH}
    \begin{aligned}
\Lag^{\mathrm{sg}}(x^\star)& \define \Lag^{\super}(x^\star)\\
&\quad\cap\left\{
d\in\Real^n \biggm| \begin{array}{rl}
	\DConCon i(x^\star)^Td= 0&\text{for}~\;i\in\ZeroBetaIndices,\\
  \DConCon i(x^\star)^Td\geq 0&\text{for}~\;i\in\BetaIndices\\
\end{array}\right\},
  \end{aligned}
\end{equation} where $\Lag^{\super}(x^\star)$ was defined in \eqref{vanishing: L(x)_super}. This cone $\Lag^{\mathrm{sg}}(x^\star)$ 
is used in the proof of Theorem~\ref{vanishing: thm: BVC-GCQ+partialLICQ} and is a 
different representation of the linearized cone $\Lag(x^\star)$ given in \eqref{eq:mpbvc-L(x)}, since on the one hand $\Lag^{\mathrm{sg}}(x^\star)\subset\Lag(x^\star)$ by definition.
On the other hand the condition $\DConCon{i}(x^\star)^Td\geq0$ for $i\in\II{0}{0}^\ConConUP(x^\star)$ in $\Lag(x^\star)$ can be equivalently written as
\begin{equation*}
	\DConCon i(x^\star)^Td= 0 \text{ for } i\in\ZeroBetaIndices \quad \text{ and } \quad
  \DConCon i(x^\star)^Td\geq 0\text{ for } i\in\BetaIndices,\\
\end{equation*}
since $\ZeroBetaIndices$ contains precisely those indices for which the linearized constraints of the index set $\II{0}{0}^\ConConUP(x^\star)$ are singular and
$\BetaIndices$ are nonsingular.
\begin{definition}[Partial BVC-LICQ with $\BetaIndices$]\label{vanishing: partial LICQ}
	Let $x^\star$ be a feasible point \newline of~\eqref{eq:mpbvc} and 
  consider $\BetaIndices\subset\II{0}{0}^\ConConUP(x^\star)$ from~\eqref{vanishing: def: betaH}.
  We say the \emph{\setword{partial BVC-LICQ}{word:partial_BVC_LICQ} with $\BetaIndices$}
  holds if for all multipliers $\mu^\ECon\in\Real^{n_\ECon}$, $\mu^\ICon\in\Real^{n_\ICon}$, $\mu^\ConConUP\in\Real^{n_\ConConUP}$,
  and $\mu^\VanConUP\in\Real^{n_\VanConUP\cdot n_\ConConUP}$ satisfying the linear combination 
  \begin{align*}
		0 &= \sum_{k\in\SetN_\ECon}\mu^\ECon_k\DECon_k(x^\star) 
    +\sum_{l\in\I\ICon(x^\star)}\mu^\ICon_l\DICon_l(x^\star)\\
    &\quad+\sum_{(i,j)\in\II+0(x^\star)} \mu^\VanConUP_{ij}\DVanCon ij(x^\star)
    +\sum_{(i,j)\in \II00(x^\star)} \mu^\VanConUP_{ij}\DVanCon ij(x^\star)\\
    &\quad+\sum_{\mathclap{i\in\II0-^\ConConUP(x^\star)}}\mu^\ConConUP_i\DConCon i(x^\star)
    +\sum_{\mathclap{i\in\II00^\ConConUP(x^\star)}}\mu^\ConConUP_i\DConCon i(x^\star)
    +\sum_{\mathclap{i\in\II0+^\ConConUP(x^\star)}}\mu^\ConConUP_i\DConCon i(x^\star)
	\end{align*}
  it follows that $\mu^\VanConUP_{ij}=0$ for all $(i,j)\in\BetaIndicesVC$.
\end{definition}
\begin{lemma}\label{vanishing: lem: BVC-LICQ implies P-LICQ}
  Let $x^\star$ be feasible point of~\eqref{eq:mpbvc} where~\ref{word:BVC_LICQ} holds. Then~\ref{word:partial_BVC_LICQ} with $\BetaIndices$ holds at $x^\star$.
\end{lemma} \begin{proof}
Since all gradients are linearly independent, all multipliers needs to be zero, especially $\mu^\VanConUP_{ij}=0$ for all $(i,j)\in\BetaIndicesVC$.
\end{proof}Note that \cite[Definition~5.4]{on_the_abadie} the whole index set of the bi-active points was used to define \emph{partial MPVC-LICQ}.
We stick to the weaker definition presented here with $\BetaIndices$ instead of $\II{0}{0}(x^\star)$,
since it still allows us to get~\ref{word:GCQ} with~\ref{word:BVC_GCQ} and partial BVC-LICQ with $\BetaIndices$.
\begin{theorem}\label{vanishing: thm: BVC-GCQ+partialLICQ}
	If $x^\star$ is a feasible point, where~\ref{word:BVC_GCQ} and~\ref{word:partial_BVC_LICQ} with $\BetaIndices$ holds, then~\ref{word:GCQ} holds at $x^\star$.
\end{theorem}
\begin{proof}
	Since $\mathcal{T}(X,x^\star)\subseteq\Lag(x^\star)$ we need to show that $\mathcal{T}(X,x^\star)^{-}\subseteq\Lag(x^\star)^{-}$ for equality
  of the polar cones. Lemma~\ref{vanishing: lem. T(x) and L(x) beta} and~\ref{word:BVC_GCQ} gives us the representation of the dual tangent cone
  \begin{align}\label{vanishing: Tx*=Schnitt}
		\mathcal{T}(X,x^\star)^{-}=\bigcap_{\PosIndices \cupdot \ZeroIndices = \II{0}{0}^\ConConUP}\Lag^{\NLP(\PosIndices,\ZeroIndices,x^\star)}(x^\star)^{-}
	\end{align} From~\cite[Theorem~2]{giorgi} we obtain the representation for dual linear cones \begin{align}\label{vanishing: LLNLPxstar*}
		&\left( \Lag^{\NLP(\PosIndices,\ZeroIndices,x^\star)}(x^\star) \right)^{-}=\bigg\{ v\in\Real^n\,\bigg|\,
		v = \sum_{k\in\SetN_\ECon}\mu^\ECon_k\DECon_k(x^\star) 
    +\sum_{l\in\I\ICon(x^\star)}\mu^\ICon_l\DICon_l(x)\notag\\
    &\hspace{14em}+\enspace\sum_{\mathclap{(i,j)\in\II{+}{0}(x^\star)\cup\PosIndicesVC}}\enspace \mu^\VanConUP_{ij}\DVanCon ij(x) 
    +\sum_{i\in\I0^\ConConUP(x^\star)}\mu^\ConConUP_i\DConCon i(x)\notag\\
  &\hspace{14em}\text{with }\mu^\ICon_l\geq0\text{ for }l\in\I\ICon(x^\star),\notag\\
  &\hspace{14em}\mu^\VanConUP_{ij}\geq0\text{ for }(i,j)\in\II{+}{0}(x^\star)\cup\PosIndicesVC,\notag\\
  &\hspace{14em}\mu^\ConConUP_i\geq0\text{ for }i\in\II{0}{+}^\ConConUP(x^\star) \cup \PosIndices~\bigg\}
	\end{align}and taking $\BetaIndices$ from~\eqref{vanishing: def: betaH} and the respective representation for $\Lag(x^\star)=\Lag^{\mathrm{sg}}(x^\star)$ from~\eqref{vanishing: LLxstar with betaH}
  it follows that \begin{align}\label{vanishing LLxstar*}
		&\Lag^{\mathrm{sg}}(x^\star)^{-} = \bigg\{ v\in\Real^n\,\bigg|\,
			v = \sum_{k\in\SetN_\ECon}\mu^\ECon_k\DECon_k(x^\star) +\sum_{l\in\I\ICon(x^\star)}\mu^\ICon_l\DICon_l(x)\notag\\
      &\hspace{12em}+\enspace\sum_{\mathclap{(i,j)\in\II{+}{0}(x^\star)}}\enspace \mu^\VanConUP_{ij}\DVanCon ij(x) +\sum_{i\in\I0^\ConConUP(x^\star)}\mu^\ConConUP_i\DConCon i(x)\notag\\ 
      &\hspace{12em}\text{with }\mu^\ICon_l\geq0\text{ for }l\in\I\ICon(x^\star),\notag\\
      &\hspace{12em}\mu^\VanConUP_{ij}\geq0\text{ for }(i,j)\in\II{+}{0}(x^\star),\notag\\
      &\hspace{12em}\mu^\ConConUP_i\geq0 \text{ for }i\in\II{0}{+}^\ConConUP(x^\star)\cup\BetaIndices\bigg\}.
	\end{align} 
  Let $v\in\mathcal{T}(X,x^\star)^{-}$ and it follows from~\eqref{vanishing: Tx*=Schnitt}
  that $v\in \left( \Lag^{\NLP(\PosIndices,\ZeroIndices,x^\star)}(x^\star) \right)^{-}$ for all
  $\PosIndices \cupdot \ZeroIndices = \II{0}{0}^\ConConUP(x^\star)$. Taking again $\BetaIndices$ we set 
  \begin{align*}
		(\hat{\mathcal{B}}^\ConConUP_1,\hat{\mathcal{B}}^\ConConUP_2)\define(\BetaIndices,\ZeroBetaIndices)
    \text{ and }(\overline{\mathcal{B}}^\ConConUP_1,\overline{\mathcal{B}}^\ConConUP_2)
    \define(\hat{\mathcal{B}}^\ConConUP_2,\hat{\mathcal{B}}^\ConConUP_1)=(\ZeroBetaIndices,\BetaIndices) 
	\end{align*} with respective $\hat{\mathcal{B}}_1,\hat{\mathcal{B}}_2$ and $\overline{\mathcal{B}}_1,\overline{\mathcal{B}}_2$. 
  Together with~\eqref{vanishing: LLNLPxstar*} it follows that there exists two vectors
	$\hat{\mu}=(\hat{\mu}^\ECon,\hat{\mu}^\ICon,\hat{\mu}^\VanConUP,\hat{\mu}^\ConConUP)$ 
  with $\hat\mu^\ICon_l\geq0\text{ for }l\in\I\ICon(x^\star),~\hat\mu^\VanConUP_{ij}\geq0$ 
  for $(i,j)\in\II{+}{0}(x^\star)\cup\hat{\mathcal{B}}_1,~\hat\mu^\ConConUP_i\geq0$ 
  for $i\in\II{0}{+}^\ConConUP(x^\star)\cup\hat{\mathcal{B}}^\ConConUP_1$, 
  and $\overline{\mu}=(\overline{\mu}^\ECon,\overline{\mu}^\ICon,\overline{\mu}^\VanConUP,\overline{\mu}^\ConConUP)$ 
  with $\overline\mu^\ICon_l\geq0\text{ for }l\in\I\ICon(x^\star),~\overline\mu^\VanConUP_{ij}\geq0$ 
  for $(i,j)\in\II{+}{0}(x^\star)\cup\overline{\mathcal{B}}_1,~\overline\mu^\ConConUP_i\geq0$ 
  for $i\in\II{0}{+}^\ConConUP(x^\star)\cup\overline{\mathcal{B}}^\ConConUP_1$
	such that
  \begin{align*}
	v = \!\sum_{k\in\SetN_\ECon}\hat\mu^\ECon_k\DECon_k(x^\star) +\!\sum_{l\in\I\ICon(x^\star)}\hat\mu^\ICon_l\DICon_l(x)
	+\enspace\sum_{\mathclap{(i,j)\in\II{+}{0}(x^\star)\cup\hat{\mathcal{B}}_1}}\enspace \hat\mu^\VanConUP_{ij}\DVanCon ij(x) 
  +~\sum_{\mathclap{i\in\I0^\ConConUP(x^\star)}}~\hat\mu^\ConConUP_i\DConCon i(x),\\
	v = \!\sum_{k\in\SetN_\ECon}\overline\mu^\ECon_k\DECon_k(x^\star) +\!\sum_{l\in\I\ICon(x^\star)}\overline\mu^\ICon_l\DICon_l(x)
	+\enspace\sum_{\mathclap{(i,j)\in\II{+}{0}(x^\star)\cup\overline{\mathcal{B}}_1}}\enspace \overline\mu^\VanConUP_{ij}\DVanCon ij(x) 
  +~\sum_{\mathclap{i\in\I0^\ConConUP(x^\star)}}~\overline\mu^\ConConUP_i\DConCon i(x).\notag
	\end{align*}
	By subtracting both representations we obtain \begin{align*}
		0&=\sum_{k\in\SetN_\ECon}(\hat\mu^\ECon_k-\overline\mu^\ECon_k)\DECon_k(x^\star)
		+\sum_{l\in\I\ICon(x^\star)}(\hat\mu^\ICon_l-\overline{\mu}^\ICon_l)\DICon_l(x)
    +~\sum_{\mathclap{i\in\I0^\ConConUP(x^\star)}}~(\hat\mu^\ConConUP_i-\overline\mu^\ConConUP_i)\DConCon i(x)\\&\quad
		+\enspace\sum_{\mathclap{(i,j)\in\II{+}{0}(x^\star)}}\enspace (\hat\mu^\VanConUP_{ij}-\overline\mu^\VanConUP_{ij})\DVanCon ij(x) 
		+\sum_{(i,j)\in\hat{\mathcal{B}}_1}\hat\mu^\VanConUP_{ij}\DVanCon ij(x) 
		-\sum_{(i,j)\in\overline{\mathcal{B}}_1}\overline\mu^\VanConUP_{ij}\DVanCon ij(x).
	\end{align*} From~\ref{word:partial_BVC_LICQ} follows that $\hat{\mu}^\VanConUP_{ij} = 0$
  for all ${(i,j)\in\hat{\mathcal{B}}_1=\BetaIndicesVC}$ and we obtain \begin{align*}
	v = \sum_{k\in\SetN_\ECon}\hat\mu^\ECon_k\DECon_k(x^\star) +\sum_{l\in\I\ICon(x^\star)}\hat\mu^\ICon_l\DICon_l(x)
	+\enspace\sum_{\mathclap{(i,j)\in\II{+}{0}(x^\star)}}\enspace \hat\mu^\VanConUP_{ij}\DVanCon ij(x) 
  +~\sum_{\mathclap{i\in\I0^\ConConUP(x^\star)}}~\hat\mu^\ConConUP_i\DConCon i(x)
	\end{align*} 
  with $\overline\mu^\ICon_l\geq0$ for $l\in\I\ICon$,~$\overline\mu^\VanConUP_{ij}\geq0$ for $(i,j)\in\II{+}{0}(x^\star)$,
  ~$\overline\mu^\ConConUP_i\geq0$ for $i\in\II{0}{+}^\ConConUP(x^\star)\cup\hat{\mathcal{B}}^\ConConUP_1$. 
  This means by~\eqref{vanishing LLxstar*} and $\hat{\mathcal{B}}^\ConConUP_1=\BetaIndices$ that $v\in\Lag^{\mathrm{sg}}(x^\star)^{-}=\Lag(x^\star)^{-}$.
\end{proof}
The following corollary is the main consequence of the preceding constraint qualification results and their relations.
\begin{corollary}\label{vanishing: cor: BVC-LICQ implies strong stat}
  If $x^\star$ is a locally optimal point where~\ref{word:BVC_LICQ} holds, then $x^\star$ is a strongly stationary point.
\end{corollary}
With Theorem~\ref{vanishing: thm: BVC-GCQ+partialLICQ} and Corollary~\ref{vanishing: cor: BVC-LICQ implies strong stat} 
we are able to verify with the fairly mild assumption (\ref{word:BVC_LICQ})  
that at any optimal point satisfying \ref{word:BVC_LICQ}, the Guignard 
constraint qualification holds. Thus strong stationarity and therefore a KKT point is obtained.

\section{A Practical Algorithm}
\label{sec:algorithm}
In the following, we generalize the algorithm given
in~\cite{flow_preprint} for solving MPVCs to the case of MPBVCs.
The key algorithmic approach is to decompose the feasible region of
the vanishing and controlling constraints into so-called
\emph{branches} according to the index sets defined
above. Specifically, at each point $x$ we distinguish for each
$i \in \SetN_\ConConUP$ whether $i \in \II{0}{-}^\ConConUP$, implying
that $\ConCon i(x) < 0$ and $\VanCon ij(x) < 0$, and its complement,
where $\ConCon i(x) \geq 0$ and $\VanCon ij(x) \geq 0$. These two
regions are connected via points $x$ where $i \in \II{0}{0}(x)$. 
The individual regions are well-behaved in the sense that normal
constraint qualifications hold with any problematic behavior showing
itself between them. We therefore begin by restricting ourselves to
one of these regions, solve one or
more subproblems within the branch, and switch branches when
necessary. In order to justify this approach, we use a projected
primal-dual gradient/antigradient flow~\cite{sequential_homotopy},
which produces a curve through the entire feasible region,
naturally leads from one branch to the other based on an augmented
Lagrangian approach, and whose equilibrium points correspond
to strongly stationary points of~\eqref{eq:mpbvc} under GCQ.

To make~\eqref{eq:mpbvc} amenable to our approach we however need to
modify the formulation from its original form. First, we convert the
nonlinear inequality constraints to a set of equality constraints and
add (possibly infinite) lower and upper bounds to all variables, which
gives rise to the convex set
\begin{equation*}
  \FS \define \{x \in \Real^{n} \mid x_{l} \leq x \leq x_{u} \}
\end{equation*}
of variable bounds. Second, we turn~\eqref{eq:mpbvc} to
its \emph{vertical form}, where
controlling and vanishing constraints correspond to individual
variables. This can of course be ensured by simply adding slack
variables and equating their values to the corresponding constraints:
We write for short $s_{\idxVC{i}{j}}\define s_{\idxVCL{i}{j}}$ for
$\idxCC{i} = 1, \ldots, n_\ConConUP$ and $j=1,\dots,n_\VanConUP$ and
order the slack variables as follows: $H_{i}(x) = s_{\idxCC{i}}$ for
$\idxCC{i}=1,\ldots,n_\ConConUP$ and
$\VanCon{i}{j}(x) = s_{\idxVC{i}{j}}$ for
$\idxCC{i}=1,\ldots,n_\ConConUP,\,j=1,\ldots,n_\VanConUP$, which
results in a slack variable vector
$s\in\Real^{(n_\VanConUP+1)n_\ConConUP}$. We also impose
bounds on $s$, producing a convex set
$\FS^\mathrm{s} \define
\{s \in \Real^{(n_\VanConUP+1)n_\ConConUP} \mid s_\mathrm{l} \leq s \leq s_\mathrm{u} \}$.
While these bounds can be trivially $\pm \infty$, it is also possible to restrict the
slack variables further to reflect bounds on the functions $H_{i}$ and
$G_{ij}$ (provided that the BVC structure is preserved).  The
resulting variant of the problem is then given by
\begin{equation}
  \label{eq:mpbvcs}
  \tag{MPBVCS}
  \begin{aligned}
    \min_{x \in \FS,s \in \FS^\mathrm{s}} \quad & f(x) \\
    \st \quad & \ECon(x) = 0, && \\
    & \ConCon{i}(x)-s_{\idxCC{i}}= 0 && \text{for } i = 1, \ldots , n_\ConConUP, \\
    & \VanCon{i}{j}(x)-s_{\idxVC{i}{j}}= 0 && \text{for } i = 1, \ldots , n_\ConConUP,~j = 1,\ldots,n_\VanConUP, \\
    & s_{\idxCC{i}}\geq0 && \text{for } i = 1, \ldots , n_\ConConUP, \\
    & s_{\idxVC{i}{j}}s_{\idxCC{i}}\geq0 && \text{for } i = 1, \ldots , n_\ConConUP,~ j = 1,\ldots,n_\VanConUP. \\
  \end{aligned}
\end{equation}
In order to extend the approach from~\cite{flow_preprint}
to~\eqref{eq:mpbvcs}, some algorithmic modifications are required,
mainly concerning the feasible set for the slack variables of each
block of vanishing constraints.  In the classic MPVC case, our
algorithm coincides with the algorithm in~\cite{flow_preprint}, thereby
facilitating the applicability of our proposed approach.

\subsection{Projected Gradient/Antigradient Flows}
Consider the nonlinear program
\begin{equation}
  \label{eq:nlp}
  \tag{NLP}
  \min_{x \in \FS} f(x) \quad \text{s.t. } c(x) = 0
\end{equation}
with constraint function $c:\Real^n\to\Real^m$ and
let $\Lag(x, y) = f(x) + \dualProd{y}{c(x)}$
and
$\Lag^{\rho}(x, y) = \Lag(x, y) + \tfrac{\rho}{2}\dualNorm{c(x)}^{2}$
denote its normal and augmented Lagrangian with $\rho > 0$ respectively.
The associated \emph{projected gradient / anti-gradient flow} is defined by
\begin{equation*}
  \begin{aligned}
    \dot{x}(t) &= P_{\mathcal{T}(\FS, x(t))} \left( -\nabla_{x} \Lag^{\rho} (x(t), y(t)) \right) &\text{ with } x(0)&=x_0 \text{, and }\\
    \dot{y}(t) &= \nabla_{y} \Lag^{\rho}(x(t), y(t)), &\text{ with } y(0)&=y_0,
  \end{aligned}
\end{equation*} with $\mathcal{T}(\FS,x(t))$ from \eqref{vanishing: T(x)} the tangent cone to $\FS$ at $x(t)$
where $P_{M}$ denotes the orthogonal projection onto  the nonempty closed convex set $M$.  
Equilibrium points of this flow coincide with stationary points of~\eqref{eq:nlp}~\cite{sequential_homotopy}.

To reach stationary points we integrate along the stiff flow with implicit schemes. 
A projected backward Euler step for $\Delta t > 0$ and $(\hat{x}, \hat{y})$ solves the system
\begin{equation}
  \label{eq:BES}
  \begin{aligned}
    x - P_{\FS}(\hat{x} - \Delta t \nabla_x \Lag^{\rho}(x, y)) &= 0, \\
    y - \hat{y} - \Delta t \nabla_y \Lag^{\rho}(x, y) &= 0,
  \end{aligned}
\end{equation}
e.g.  by a semismooth Newton (SSN) method~\cite{semismooth, ssn}.  
Equivalently, the solution of the system \eqref{eq:BES} corresponds a first-order optimal solution of the regularized subproblem
\begin{equation*}
  \begin{aligned}
    \min_{x \in \FS, w \in \Real^{m}} \; & f^{\rho}(x) + \lambda \!\left[\tfrac{1}{2}\primalNorm{x-\hat{x}}^{2} + \tfrac{1}{2}\dualNorm{w-\hat{y}}^{2}\right] \\
    \text{s.t. } & c(x) + \lambda w = 0,
  \end{aligned}
\end{equation*} where $f^{\rho}(x) \define f(x) + \tfrac{\rho}{2} \dualNorm{c(x)}^{2}$ and $\lambda\define 1/\Delta t$. 
A clear advantage of this subproblem is that it is always feasible and satisfies LICQ~\cite{sequential_homotopy}.
We can therefore solve it using a standard solver such as~\textsc{Ipopt}~\cite{ipopt}.
\subsection{Algorithmic Framework}
The practical implementation of our approach involves identifying convex sets based on the index sets of the slack variables 
and their changes, decomposing the flow into piecewise branch-consistent segments, and solving suitable reduced subsequent nonlinear problems.
\paragraph{Branch Decomposition}
We decompose the feasible region of the slack variables corresponding to
$\ConCon{i}(x) \ge 0$ and $\VanCon{i}{j}\ConCon{i} \ge 0$  for each pair $(i,j)$ into \begin{equation*}
\begin{aligned}
    \FS^\ConConUP_{+}(i) &\define \{s_{\idxCC{i}}\in \Real\,\mid\, s_{\idxCC{i}} \geq 0 \}, & \FS^\VanConUP_{+}(i,j) &\define \{s_{\idxVC{i}{j}}\in \Real\,\mid\, s_{\idxVC{i}{j}} \geq 0 \},\\
    \FS^\ConConUP_{0}(i) &\define \{s_{\idxCC{i}}\in \Real\,\mid\, s_{\idxCC{i}} = 0  \}, &\FS^\VanConUP_{0}(i,j) &\define \{s_{\idxVC{i}{j}}\in \Real\,\mid\, s_{\idxVC{i}{j}} \leq 0 \}.
  \end{aligned}
\end{equation*}We define two convex branches
\begin{equation}\label{eq: convex pieces}
  \begin{aligned}
    \FS_{+}(i,j) &\define \FS^\ConConUP_{+}(i) \times \FS^\VanConUP_{+}(i,j), \text{ and} \\
    \FS_{0}(i,j) &\define \FS^\ConConUP_{0}(i) \times \FS^\VanConUP_{0}(i,j),
  \end{aligned}
\end{equation}
where we call $ \FS_{+}(i,j)$ the \emph{upper branch} and $ \FS_{0}(i,j) $ the \emph{lower branch}. 
Note that the definition of $\FS_{+}(i,j)$ and $\FS_{0}(i,j)$ coincides
with the upper and lower branch definition presented in~\cite{flow_preprint}
for a block size of one ($n_\VanConUP = 1$).
The intersection of lower and upper branch consists of the bi-active point.
We introduce a \emph{signature} $\signature \in \{0, 1\}^{n_\ConConUP \times n_\VanConUP}$ and 
define the set of variable bounds of the slack variables of block $i$ as
\begin{align}\label{eq: convex pieces BVC}
    \FS^\mathrm{\BVC}(\signature,i) \define& \left(\bigcap_{j=1}^{n_\VanConUP}\begin{cases}
   \FS^\ConConUP_{+}(i) &\text{ if } \signature_{i,j} = 1, \\
  \FS^\ConConUP_{0}(i) &\text{ if } \signature_{i,j} = 0,
\end{cases}\right) \\
&\quad\times \left(\bigtimes_{j=1}^{n_\VanConUP} 
    \begin{cases}
   \FS^\VanConUP_{+}(i,j) &\text{ if } \signature_{i,j} = 1, \\
  \FS^\VanConUP_{0}(i,j) &\text{ if } \signature_{i,j} = 0
\end{cases}\right).\notag
\end{align}
Note that this set enforces $s_{\idxCC{i}} = 0$ if the
block corresponding to $\idxCC{i}$ contains any vanished constraints.
We employ this
definition in order to define the feasible set of the variables
in \eqref{eq:mpbvcs} restricted to the branches associated with
$\signature \in \{0,1\}^{n_\ConConUP\times n_\VanConUP}$ as
\begin{equation*}
\FS^\mathrm{b}(\signature) \define \FS^\mathrm{s} \cap \left( \bigtimes_{i = 1}^{n_\ConConUP} \FS^\mathrm{\BVC}(\signature,i),
  \right)
\end{equation*}
leading to the following nonlinear program associated with $\signature$:
\begin{equation}
  \label{eq:branch_nlp}
  \tag{NLP$\sigma$}
    \begin{aligned}
    \min_{x \in \FS, s \in \FS^\mathrm{b}(\signature)} \quad & f(x) \\
    \st \quad & g(x) = 0,  &&\\
    & \ConCon{i}(x)-s_{\idxCC{i}}= 0, && \text{for } i = 1, \ldots , n_\ConConUP,\\
    & \VanCon{i}{j}(x)-s_{\idxVC{i}{j}}= 0, && \text{for } i = 1, \ldots , n_\ConConUP,~j = 1,\ldots,n_\VanConUP \\
  \end{aligned}
\end{equation}
with convex domain $\FS \times \FS^\mathrm{b}(\signature)$.
As in the case of MPVC, a stationary point of~\eqref{eq:mpbvcs} is stationary for all~\eqref{eq:branch_nlp}, but not vice versa, which can be proven analogously to~\cite{flow_preprint}.

\paragraph{Piecewise Gradient Flows}
Starting from a initial guess $(x_0,s_0,y_0)$ with $x_0 \in \FS$ and $s_0 \in \FS^\mathrm{b}(\signature)$ for a given signature $\signature$, 
we consider the flow
\begin{equation*}
  \begin{aligned}
  (\dot{x}(t),\dot{s}(t)) &= P_{T(\FS\times \FS^\mathrm{b}(\signature), (x(t),s(t)))} \left( -\nabla_{x,s} \Lag^{\rho} (x(t), s(t), y(t)) \right),\\
  &\hspace{1em}\text{ with }x(0)=x_0,~s(0)=s_0 \text{ and }\\
    \dot{y}(t) &= \nabla_{y} \Lag^{\rho}(x(t), s(t), y(t))\text{ with }y(0)=y_0
  \end{aligned}
\end{equation*} for solving~\eqref{eq:branch_nlp}. If no signature is given, 
we refer the reader to use the branch initialization strategy described in~\cite{flow_preprint}.
Since an optimal point of~\eqref{eq:branch_nlp} is in general not a strongly stationary point,
we need to switch between convex sets at bi-active indices of $s(t)$ for some $t>0$ by flipping one or more branches,
yielding a continuous (but generally non-smooth) piecewise flow. This switching mechanism is combined with a stationarity 
test for~\eqref{eq:mpbvcs} to detect termination.
Specifically, we test stationarity according to Theorem~\ref{vanishing: thm: kkt}. 
As described in~\cite{flow_preprint}, the multipliers $\eta^\ConConUP$ and $\eta^\VanConUP$ are chosen to minimize the absolute value of the corresponding components of
$\xi^{\MPBVCS} \define - \nabla_{x,s} \Lag^{\MPBVCS}(x,s,y)$ subjecting to their sign constraints. We terminate once the
distance
\begin{equation}\label{alg:distance}
  \Delta \xi^{\MPBVCS} \define \xi^{\MPBVCS} - P_{T^{-}(\FS\times \FS^\mathrm{s}, (x^{-},s^{-}))}(\xi^{\MPBVCS})
\end{equation}
falls below a threshold $\epsilon$ in an appropriate norm and therefore stationarity is (approximately) reached.
Otherwise we change the signatures from $\signature$ to $\signature'$ based on
the signs of the negative gradient $-\nabla_{x,s} \Lag^{\rho}(x,s,y)$.
To this end, we iterate over the indices $(i, j) \in \SetN_\ConConUP\times\SetN_\VanConUP$ and make the following
changes:
\begin{itemize}
\item
  If $(i,j) \in \IndexSet_{00}(x)$, $\signature_{i,j}=1$, and
  \begin{equation*}
    (-\nabla_{s} \Lag^{\rho}(x,s,y))_{\idxCC{i}}\leq 0 \quad\textrm{and}\quad
    (-\nabla_{s} \Lag^{\rho}(x,s,y))_{\idxVC{i}{j}} < 0,
  \end{equation*}
  then we switch to the lower branch $\FS_0(i,j)$ by setting $\signature'_{i,j} = 0$.
\item
  If $(i, j) \in \IndexSet_{00}(x)$, $\signature_{i,j} = 0$, and
  \begin{equation*}
    (-\nabla_{s} \Lag^{\rho}(x,s,y))_{\idxVC{i}{j}} > 0,
  \end{equation*} or
  \begin{equation*}
    (-\nabla_{s} \Lag^{\rho}(x,s,y))_{\idxVC{i}{j}} = 0 \quad \text{and}\quad (-\nabla_{s} \Lag^{\rho}(x,s,y))_{\idxCC{i}} > 0,
  \end{equation*}
  then we switch to the upper branch $\FS_+(i,j)$ by setting $\signature'_{i,j} = 1$.
\item
  All other signature values are kept constant.
\end{itemize}
Since the changes in signatures occur only in the bi-active variables, it follows that $s \in \FS^\mathrm{b}(\signature) \cap \FS^\mathrm{b}(\signature')$.
We can therefore continue the piecewise flow through the convex piece associated with $\signature'$
starting from $(x,s)$ and $y$ by integrating the flow along the signature-modified initial value problem
\begin{equation*}
\begin{aligned}
  (\dot{x}(t),\dot{s}(t)) &= P_{T(\FS\times \FS^\mathrm{b}(\signature'), (x(t),s(t)))} \left( -\nabla_{x,s} \Lag^{\rho} (x(t), s(t), y(t)) \right) , \quad\\
      &\hspace{1em}\text{ with }x(0)=x,~s(0)=s \text{ and }\\
    \dot{y}(t) &= \nabla_{y} \Lag^{\rho}(x(t), s(t), y(t))\text{ with }y(0)=y.
  \end{aligned}
\end{equation*} 
Note that $\dot{y}(t)$ is kept as before.

\paragraph{Reduced Subsequent NLPs}
The lower branch $\FS_0$ in~\eqref{eq: convex pieces} corresponds to indices $\II{0}{-}$, 
for which the controlling constraints act as equalities and the associated 
vanishing constraints are inactive with multipliers $\mu_{ij}^\VanConUP=0$ for $(i,j) \in \II{0}{-}$. 
Motivated by the structure of the linearized cone~\eqref{eq:mpbvc-L(x)}, we omit from the subproblem the vanishing constraints 
$\VanCon{i}{j}$ and the corresponding slack variables $s_{\idxVC{i}{j}}$ for $(i,j)\in\II{0}{-}$, since they do not affect the 
optimality conditions and the associated controlling constraints must remain active at every feasible point.
Let $\signature$ be a signature for the current branch decomposition. Vanishing Constraints in the index set $\II{0}{-}$ are 
located in the lower branch. Let
\begin{equation*}
  \IUBIndices(\signature) \define\{i\in\{1,\dots,n_\ConConUP\}\,\mid\,\signature_{i,j} = 1\text{ for all }j\in\{1,\dots,n_\VanConUP\}\}
\end{equation*}
denote the index set detecting blocks with no vanishing constraint in the lower branch, since only vanishing constraints where the whole block
lies in the upper branch can contribute to optimality. Based on~\cite{flow_preprint}, the reduced slack variables are defined as follows: 
\begin{equation*}
  \tilde{s} \define \big(s_{\idxCC{1}},\ldots,s_{n_\ConConUP}, (s_{\idxVC{i}{j}},\;\text{for}\;\signature_{i,j} = 1,\,i\in\IUBIndices(\signature))\big)^T,
\end{equation*}
with the reduced feasible set for each block $i$ of vanishing constraints 
\begin{align*}
    \FS^\mathrm{\BVC red}(\signature,i) \define& \left(\bigcap_{j=1}^{n_\VanConUP}\begin{cases}
   \FS^\ConConUP_{+}(i) &\text{ if } \signature_{i,j} = 1 \\
  \FS^\ConConUP_{0}(i) &\text{ if } \signature_{i,j} = 0.
\end{cases}\right) \\
&\quad\times \left(\bigtimes_{j=1}^{n_\VanConUP} \FS^\VanConUP_{+}(i,j) \text{ if } i\in\IUBIndices(\signature) \right).
\end{align*}
The feasible set for the reduced slack variables is
\begin{equation*}
  \FS^\mathrm{red}(\signature) \define \FS^\mathrm{s} \cap \left(\bigtimes_{i = 1}^{n_\ConConUP}  \FS^\mathrm{\BVC red}(\signature,i) \right).
\end{equation*}
The reduced constraints are
\begin{equation*}
  \tilde{c}(x,\tilde{s}) \define
  \begin{pmatrix*}[l]
      \multicolumn{2}{c}{g(x)} \\
      \ConCon{i}(x)-s_{\idxCC{i}} & \text{for } i=1,\ldots,n_\ConConUP\\
      \VanCon{i}{j}(x)-s_{\idxVC{i}{j}} & \text{for } \signature_{i,j}=1,\,i\in\IUBIndices(\signature)
    \end{pmatrix*} \in \Real^r
\end{equation*}
with $r = m+n_\ConConUP+\#\{(i,j) \mid \signature_{i,j} = 1\}$. We obtain the reduced subproblem:
\begin{equation}\label{eq: Red-NLP}
  \tag{Red-NLP$\sigma$}
  \begin{aligned}
    \min_{\substack{x \in \FS \\ \tilde{s}\in \FS^\mathrm{red}(\signature) \\ \tilde{w} \in \Real^r}} & \: \: \tilde{f}^{\rho}(x,\tilde{s})
    + \lambda \left[ \tfrac{1}{2} \primalNorm{x - \hat{x}}^{2} 
        + \tfrac{1}{2} \|\tilde{s} - \tilde{\hat{s}}\|_{\tilde{S}}^2 
        + \tfrac{1}{2} \dualNorm{\tilde{w}- \tilde{\hat{y}}}^2 
    \right] \\
    \text{s.t.} & \: \: \tilde{c}(\tilde{x}) + \lambda \tilde{w} = 0,
  \end{aligned}
\end{equation} with real Hilbert space $\tilde{S}$ equipped with its induced norm 
for the reduced slack variables $\tilde{s}$.
Furthermore, \( \tilde{w} \) is a vector of the same dimension as \( \tilde{c} \), 
\( \tilde{f}^{\rho}(x,\tilde{s}) \) denotes the reduced augmented objective, 
and \( \hat{x}\), \(\tilde{\hat{s}}\), and \( \tilde{\hat{y}} \) denote 
the corresponding reduced solutions from the previous iteration. Analogously to~\cite{flow_preprint}, 
after solving~\eqref{eq: Red-NLP}, we update the slack variables \( s_{\idxVC{i}{j}} \) for \( i \not\in\IUBIndices(\signature)\) 
such that they satisfy  $\VanCon{i}{j}(x)-s_{\idxVC{i}{j}}=0$ and project the primal solution onto \( \FS^\mathrm{s}\cap \FS_0^\VanConUP(i,j) \). 
We derive the variables $y$ by $\tilde{\hat{y}} - \tilde{w}$ and assign zero multipliers to the VCs for \( i \not\in\IUBIndices(\signature)\).

\paragraph{Summary}
The following Theorem \ref{vanishing: alg thm} and Algorithm \ref{vanishing: algorithm} summarizes the results introduced in this section. 
The theorem states the theoretical foundation of the proposed method 
that equilibria of the piecewise gradient flow approach are strong stationary points.
\begin{theorem}\label{vanishing: alg thm}
	A strongly stationary point is equivalent to a point of zero flow under
  the proposed strategy based on piecewise flows, reduced subproblems, and switching strategy.
\end{theorem}
\begin{proof}
Analogous to the proof of Theorem 4.1 in~\cite{flow_preprint}.
\end{proof}
The method is summarized as follows:
\begin{algorithm}
  \DontPrintSemicolon
  \caption{Gradient Flow approach for MPBVC}\label{vanishing: algorithm}
  \KwIn{Primal guess $x_0$, optional dual guess $y_0$, initial $\rho, \lambda$, parameters $\lambda_{\mathrm{inc}}, N_\mathrm{max}$
  and tolerance $\epsilon$.}
  \KwOut{Solution $x,s,y$}
  $x \gets x_0$, $y \gets y_0$\;
  Determine branch configuration $\FS_0 | \FS_+$ and initialize $s$\;
  \While{$k < N_\mathrm{max}$}{
    Solve reduced Problem~\eqref{eq: Red-NLP}\;
    
    Update $y$ and $s$ from solution of~\eqref{eq: Red-NLP}\;
    
    \IfU{\(\Vert \Delta \xi^{\MPBVCS}\Vert < \epsilon\) \text{from~\eqref{alg:distance}}}{
      \Return{$x,s,y$}\;
    }
    \Else{
      Update branch configuration $\FS_0 | \FS_+$\;
      Update $\rho$ and $\lambda$ with step increasing factor $\lambda_{\mathrm{inc}}$ according to \cite{flow_preprint}\;
    }
  }
  \Return{$x,s,y$}\;
\end{algorithm}

\section{Numerical Experiments}\label{sec:num_ex}

We demonstrate the effectiveness of our approach
on several test instances of MPBVCs. Since our approach coincides with~\cite{flow_preprint} in the MPVC case, we only show MPBVC with $n_\VanConUP>1$.
We will discuss an academic example and truss topology design problems from~\cite[Sec.~9.5]{diss_hoheisel},
which are adapted for block vanishing constraints with two or three load cases. 
All experiments were performed on an \texttt{Intel Core i7-14700} clocked
at \SI{2.1}{\giga\hertz} based on an implementation in \texttt{Python 3.13.5} with \texttt{CasADi 3.7.1}~\cite{casadi} and
using \texttt{Ipopt 3.14.11}~\cite{ipopt} as subproblem solver.
 
\subsection{An Academic Example}
We begin by examining the following small MPBVC academic example in two variables and two blocks with two VCs each:
\begin{mini}|s|
	{x_1,x_2}{\exp(x_1)x_2^2}
	{\label{vanishing: MPBVC-Example2}}
	{}
		\addConstraint{\big(-x_1^2x_2 +2x_2+4\big)\big(x_1+x_2+2\big)}{\geq0,}{\quad\text{Index}~\VanConUP:(1,1)}
		\addConstraint{\big(-\frac{1}{2}x_2^2+x_2\big)\big(x_1+x_2+2\big)}{\geq0,}{\quad\text{Index}~\VanConUP:(1,2)}
		\addConstraint{x_1+x_2+2}{\geq0,}{\quad\text{Index}~\ConConUP: ~1}
		\addConstraint{\big(x_1-1\big)\big(x_1x_2-2\big)}{\geq0,}{\quad\text{Index}~\VanConUP:(2,1)}
		\addConstraint{\big(x_2-1.2599\big)\big(x_1x_2-2\big)}{\geq0,}{\quad\text{Index}~\VanConUP:(2,2)}
		\addConstraint{x_1x_2-2}{\geq0,}{\quad\text{Index}~\ConConUP: ~2}
\end{mini}
The optimal point is at $x^\star = (1.5874,1.2599)$ with the index sets
 \begin{align*}
   \II{+}{+} &= \{(1,1),(2,1)\}, &\II{+}{0} &= \{(1,2) \},\\
   \II{0}{+} &= \{(2,1) \}, &\II{0}{-} &= \{(2,2)\},
\end{align*}
and nonzero Lagrange multipliers $\mu_2^\ConConUP = 7.2301$ and $\mu_{12}^\VanConUP = 0.8475$.
The feasible set is non convex due to the active controlling constraint $\ConCon{2}$ and the vanishing constraint $\VanCon{2}{1}$.
Figure~\ref{fig:num_ex:academic_flow} illustrates the continuous flow
generated by our approach using $\lambda=100$ and a step factor
$\lambda_\mathrm{inc}=1.01$ for visualization purposes, while
Figure~\ref{fig:num_ex:academic_discrete_flow} shows its discrete
approximation. Using the standard parameter values $\lambda=0.1$ and
$\lambda_\mathrm{inc}=2$, the method converges to the optimal point in seven
iterations from both starting points $\xi_0=(1,4.5)$ and $\xi_1=(5,3)$. The
numbers of \textsc{Ipopt} subproblem iterations are \num{51} and
\num{61}, respectively.
\begin{figure}[h]
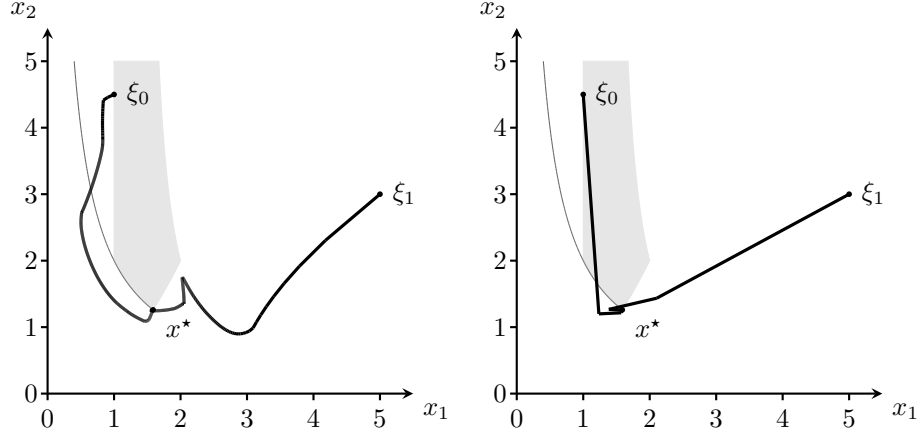

  \centering
    \begin{subfigure}[t]{0.499\textwidth}
        \centering\includegraphics[width=\textwidth]{academic_feas}
        \caption{Continuous flow (spikes show branch switches)}
        \label{fig:num_ex:academic_flow}
    \end{subfigure}~
    \begin{subfigure}[t]{0.499\textwidth}
        \centering\includegraphics[width=\textwidth]{academic_feas_short}
        \caption{Discrete flows (seven iterations with standard settings)}
        \label{fig:num_ex:academic_discrete_flow}
    \end{subfigure}
    \caption{
      Feasible set (gray), active controlling constraints (dark gray), and piecewise flows (emanating from
      the two starting points ${\xi_0=(1,4.5)}$ and ${\xi_1=(5,3)}$) for the academic example.}
    \label{fig:num_ex:academic}
  \end{figure}
\subsection{Truss Topology Optimization}
Truss topology optimization problems arise in civil engineering, 
where structures must withstand prescribed loads while minimizing material usage 
(see~\cite{topology_optimization}) with typical applications including bridge and cantilever designs.
We consider a problem variant where material is placed between fixed positions
in two dimensions with the goal of minimizing the total material volume
while ensuring that compliance bounds are met under different loads. This problem can be formulated as
\begin{equation*}
  \begin{aligned}
    \min_{a \in \Real^{N},\, u \in \Real^{Ld}} \quad & \sum_{i=1}^{N} \ell_i a_i \\
    \text{s.t.} ~\quad\quad&K(a) u_l = f_l, && l = 1,\ldots,L, \\
    &f_l^{T} u_l \le c, && l = 1,\ldots,L, \\
    & 0 \le a_i \le \overline{a}, && i = 1,\ldots,N, \\
    &\hspace{-5em} (-\sigma_{il}^2(a,u) +(\overline{\sigma}+\underline{\sigma})\sigma_{il}(a,u) - \underline{\sigma}\overline{\sigma})\, a_i \ge 0, && i = 1,\ldots,N,\; l = 1,\ldots,L.
  \end{aligned}
\end{equation*}
Here, $a_i$ denotes the cross-sectional area of bar $i$ of length $\ell_i > 0$, 
bounded by $0 \le a_i \le \overline{a}$. These areas naturally give rise to the objective
of minimizing the volume of the resulting structure.
Each load case $l$ produces displacements $u$ according to an equilibrium condition $K(a)u_l = f_l$, where $f_l \in \Real^d$ are external
forces acting on the free nodes, and $K(a)$ is the global stiffness matrix
\[
  K(a) = \sum_{i=1}^{N} a_i \frac{E}{\ell_i} \gamma_i \gamma_i^{T},
\]
with $E$ the Young's modulus and $\gamma_i \in \Real^{d}$ containing the directional cosines of bar $i$.  
The compliance constraints $f_l^{T}u_l \le c$ limit the structure's flexibility by a constant $c>0$.
The stress in bar $i$ under load case $l$ is
\[
  \sigma_{il}(a,u) = E \frac{\gamma_i^{T}u_l}{\ell_i},
\]
which must satisfy compression and tension bounds $\underline{\sigma} \le \sigma_{il}(a,u) \le \overline{\sigma}$. This condition must, however, only hold when $a_{i} > 0$, which can
be modeled by the vanishing constraint
\[
  (\overline{\sigma} - \sigma_{il}(a,u))\cdot(\sigma_{il}(a,u)-\underline{\sigma})= -\sigma_{il}^2(a,u) +(\overline{\sigma}+\underline{\sigma})\sigma_{il}(a,u) - \underline{\sigma}\overline{\sigma}.
\]
Hence, the formulation is an MPBVC. Note that
for $\-\underline{\sigma}=\overline{\sigma}$ this expression simplifies to
$(\overline{\sigma}^2 - \sigma_{il}(a,u)^2)$, \ie a
well-known~\cite{diss_hoheisel} formulation. We only
note $\overline{\sigma}$ if $-\underline{\sigma} = \overline{\sigma}$.

All instances are normalized by setting $E \equiv 1$ for comparability. The forces for the two load case instances 
are taken from the Hook examples in~\cite{diss_hoheisel}, after removing redundant bars. In addition to the benchmark 
instances from the literature, we also consider for the first time in the MPVC literature truss topology 
optimization problems with three load cases ($L=3$). Further details on the problem 
setup and the initialization can be found in~\cite{flow_preprint}. The resulting benchmark problems are summarized 
in Table~\ref{tab:num_ex:problemslist}.
\begin{table}[ht]
  \caption{List of problem instances. Changing sizes and parameters are visible. $L$ corresponds to the number of load cases.
  \label{tab:num_ex:problemslist}}
  \begin{adjustbox}{max width=\textwidth}
  \centering\begin{tabular}{l S[table-format=1] S[table-format=3] S[table-format=4] S[table-format=4] S[table-format=4] S[table-format=3.1]  S[table-format=3] S[table-format=1.1] S[table-format=2.1]}
     \toprule
     \textbf{Problem} & $L$& \textbf{$\#$bars} &\textbf{$\#$vars} & \textbf{$\#$Cons} &\textbf{$\#$VC}  & $\overline{a}$ & $c$ & ~$\overline{\sigma}$~ & ~$\underline{\sigma}$~\\
     \midrule
    \instName{TenBar1}& 2 &  10 &  48 &  38 &  20 & 100.0 &  10 &   1.0 & -1.0\\ 
    \instName{TenBar2}& 3 &  10 &  67 &  57 &  30 & 100.0 &  10 &   1.0 & -1.0\\ 
    \instName{Cant1} &  2 & 224 & 770 & 546 & 448 &   1.0 & 100 &   2.2 & -2.2\\ 
    \instName{Cant2} &  2 & 224 & 770 & 546 & 448 &   1.0 & 100 &   2.2 & -1.2\\ 
    \instName{Hook1} &  2 & 661 &2173 &1512 &1322 & 100.0 & 100 &   3.5 & -3.5\\ 
    \instName{Hook2} &  2 & 661 &2173 &1512 &1322 & 100.0 & 100 &   3.0 & -3.0\\ 
    \instName{Hook3} &  3 & 661 &2929 &2268 &1983 & 100.0 & 100 &   3.5 & -3.5\\ 
    \instName{Hook4} &  3 & 661 &2929 &2268 &1983 & 100.0 & 100 &   3.0 & -1.7\\ 
    \instName{Hook5} &  3 & 661 &2929 &2268 &1983 &   0.7 & 100 &   3.0 & -1.7\\ 
    \bottomrule
  \end{tabular}
\end{adjustbox}
\end{table}
\FloatBarrier
\paragraph{Ten-bar Truss}
The first design problem we examine is the so-called \emph{ten-bar} truss problem, as presented in~\cite{diss_hoheisel},
and depicted in Figure~\ref{fig:num_ex:tenbar_potential}. This problem consists of ten potential bars which are positioned between six points. 
Among these points, two are fixed to a wall, while a loading force pointing downward is applied at the lower right corner.
The parameters for this problem are set as follows: $(\overline{a}, c, \overline{\sigma}) = (\num{100.0}, \num{10.0}, \num{1.0})$.

In~\cite{flow_preprint} it was already shown that we can apply the presented approach to single load cases.
Therefore, we extend the ten-bar problem to a problem with two load cases \instName{TenBar1} and three load cases \instName{TenBar2}.
The results are visible in Figure~\ref{fig:num_ex:tenbar}. Instance \instName{TenBar1} has the same structure as the ten-bar problem with one load case. 
The bars themselves are thicker, which leads to a higher volume of \num{9.0}.
For \instName{TenBar2} there are two additional bars needed and therefore the volume increased.
The instance \instName{TenBar2} is a MPBVC with three VCs in each block, which has not shown in Literature before.
The detailed results of \instName{TenBar2} are noted in Table~\ref{tab:num_ex:tenbar}. Observe that the stress values are inside the bounds. 
All three compliance constraints are inactive, which coincides with the single load case observations in~\cite{flow_preprint} and~\cite{Hoheisel22}.
Here, $\rho=0.2$ was used to reduce the overall iterations to \num{10} and, respectively \num{12}. 
The number of total subproblem iterations are \num{258} and \num{344} of \textsc{ipopt}.
Since the problem size increases with the number of load cases, a longer 
computation time and a higher number of iterations is unsurprising.
\begin{figure}[ht]
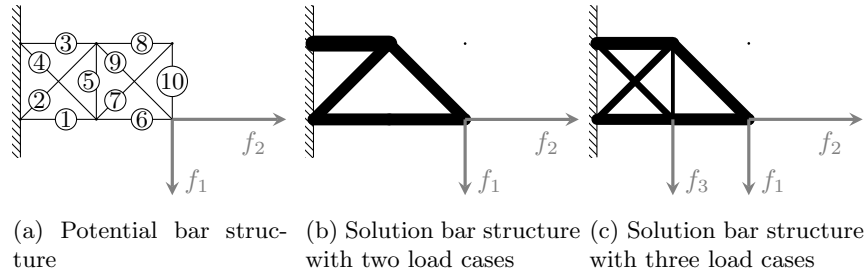

  \centering
    \begin{subfigure}[t]{0.3\textwidth}
        \includegraphics[width=\textwidth]{tenbar_potential}
        \caption{Potential bar structure}
        \label{fig:num_ex:tenbar_potential}
      \end{subfigure}~
    \begin{subfigure}[t]{0.3\textwidth}
        \includegraphics[width=\textwidth]{TenBar1}
        \caption{Solution bar structure with two load cases}
        \label{fig:num_ex:tenbar_solution}
      \end{subfigure}~\begin{subfigure}[t]{0.3\textwidth}
        \includegraphics[width=\textwidth]{TenBar2}
        \caption{Solution bar structure with three load cases}
        \label{fig:num_ex:tenbar_solution2}
      \end{subfigure}~
    \caption{Potential bars and solution structures of the \instName{TenBar1} and \instName{TenBar2} instances. \instName{TenBar1} visually coincide with e.g.~\cite{diss_hoheisel,Hoheisel22}.
    The entire structure is fixed to the gray wall on the left. The load forces are marked in dark gray.}
    \label{fig:num_ex:tenbar}
\end{figure}
\begin{table}
  \caption{Results of the \instName{TenBar2} instance for the controlling constraints $a_i$, the stress values $\sigma_{il}$, the displacement variables $u_{lj}$ for the load cases $l=1,2,3$.
  We also show the compliance $f_l^Tu_l,~i=1,2,3$ and optimized volume $V^{\star}$.}
  \label{tab:num_ex:tenbar}
  \begin{adjustbox}{max width=\textwidth}
  \centering\begin{tabular}{S[table-format=2] S[table-format=1.6] S[table-format=1.6] S[table-format=1.6] S[table-format=1.6] | S[table-format=1] S[table-format=1.6] | S[table-format=1] S[table-format=1.6] | S[table-format=1] S[table-format=1.6]}
     \toprule
     {$i$} & {$a_i^{\star}$} & {$\sigma_{i1}(a^{\star},u^{\star})$} & {$\sigma_{i2}(a^{\star},u^{\star})$} & {$\sigma_{i3}(a^{\star},u^{\star})$} & {$1j$} & {$ u_{1j}^{\star}$} & {$2j$} & {$ u_{2j}^{\star}$} & {$3j$} & {$ u_{3j}^{\star}$}\\
     \midrule
     1 & 1.39214 & -0.98193 & 1.00000 &-0.42166 &1&-0.98193 & 1 & 1.00000 & 1 &-0.42167 \\
     2 & 0.89522 & -1.00000 & 0.17039 &-0.65241 &2&-2.23228 & 2 & 0.63250 & 2 &-2.42167 \\
     3 & 1.63302 &  1.00000 &-0.06605 & 0.25289 &3& 1.00000 & 3 &-0.06605 & 3 & 0.25290 \\
     4 & 0.83017 &  0.62517 & 0.18375 & 1.00000 &4&-3.00000 & 4 & 0.40685 & 4 &-1.55771 \\
     5 & 0.47801 & -0.76772 &-0.22565 & 0.86395 &5&-1.64860 & 5 & 2.00000 & 5 &-0.42167 \\
     6 & 1.49999 & -0.66667 & 1.00000 & 0       &6&-7.64860 & 6 & 2.47290 & 6 &-2.23228 \\
     7 & 0       & -0.99876 &-0.02150 & 0.04248 &7& 1.99852 & 7 & 0.09109 & 7 & 0.34965 \\
     8 & 0       &  0.99829 & 0.15715 & 0.09675 &8&-7.21053 & 8 & 1.49890 & 8 &-3.10803 \\
     9 & 1.41421 &  1.00000 & 0       & 0       &\multicolumn{2}{c}{$f_1^Tu_1^{\star} = \num{7.64860} $} & \multicolumn{2}{c}{$f_2^Tu_2^{\star} = \num{3.00000} $} & \multicolumn{2}{c}{$f_3^Tu_3^{\star} = \num{2.42167} $} \\
     10 & 0      &  0.43856 &-0.97400 &-0.87575 &\multicolumn{6}{c}{$V^{\star} = \num{9.44323}$} \\
     \bottomrule
  \end{tabular}
\end{adjustbox}
\end{table}

\FloatBarrier

\paragraph{Cantilever Problem}

The second truss design problem is the design of a cantilever arm. 
The potential bar structure of \instName{Cant1} and \instName{Cant2} is illustrated in Figure~\ref{fig:num_ex:cantilever_potential} 
and consists of \num{27} points arranged in three rows. 
As in the previous example, the structure is fixed on the left side to a wall, and two loading forces are applied in the lower right corner. 
Potential bars connect all pairs of points, with long bars overlapping shorter ones being removed. 
This results in \num{224} potential bars. Following the approach in~\cite{diss_hoheisel},
the parameter set $(\overline{a}, c, \overline{\sigma}) = (\num{1.0}, \num{100.0}, \num{2.2})$ for instance \instName{Cant1} is employed.
For instance \instName{Cant2} we change $\underline{\sigma}$ from $-2.2$ to $\underline{\sigma}=-1.2$, 
since some material could be better in tension compared to their compression capability.
It was demonstrated in~\cite{diss_hoheisel} and~\cite{flow_preprint} that, with the parameter set of \instName{Cant1},
vanished constraints occur in the lower branch of the solution for a single load case.

For both instances we used $\rho=10^{-3}$ and $\lambda_\mathrm{inc}=1.9$.
The solution of \instName{Cant1} is depicted in Table~\ref{tab:num_ex:problemssol} and is
similar to those of the cantilever instance \instName{Cant2a} of~\cite{flow_preprint}. The volume of \num{23.9741} for the single load case increased to \num{23.88549}
in the double load case. In Figure~\ref{fig:num_ex:cantilever_solution} the solution bar structure is shown. 
As expected, the constraint on the maximal stress value for appearing bars is active.
To compute \instName{Cant2} we used the primal solution of \instName{Cant1}.
Here the volume again increased to \num{33.68315} and the resulting structure contains more bars in the upper right to satisfy 
the compression constraint of $\underline{\sigma}=-1.2$ (see Figure~\ref{fig:num_ex:cantilever_solution_2}). 
Since the instance \instName{Cant2} is initialized with the primal solution of \instName{Cant1} the number of iterations and \textsc{ipopt} subproblem
iterations decreased from \num{14} to \num{13} and \num{1497} to \num{1266} respectively. 
In both instances \instName{Cant1} and \instName{Cant2} the maximal value for the cross sectional area $\overline{a}=1$ is attained.

\begin{figure}[ht]
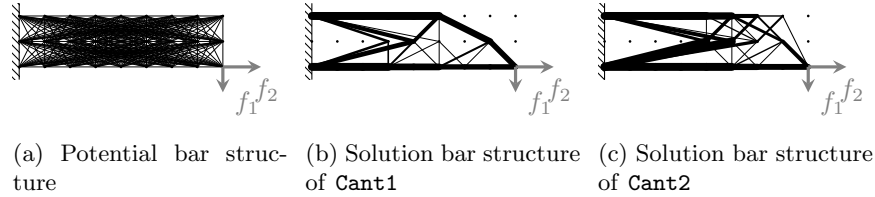

  \centering
    \begin{subfigure}[t]{0.3\textwidth}
        \includegraphics[width=\textwidth]{cantilever_potential}
        \caption{Potential bar structure}
        \label{fig:num_ex:cantilever_potential}
      \end{subfigure}~
    \begin{subfigure}[t]{0.3\textwidth}
        \includegraphics[width=\textwidth]{cantilever_2}
        \caption{Solution bar structure of \instName{Cant1}}
        \label{fig:num_ex:cantilever_solution}
      \end{subfigure}~
    \begin{subfigure}[t]{0.3\textwidth}
        \includegraphics[width=\textwidth]{cantilever_3}
        \caption{Solution bar structure of \instName{Cant2}}
        \label{fig:num_ex:cantilever_solution_2}
      \end{subfigure}
      \caption{Potential structure of the cantilever problem in (a) and the solutions in (b) and (c) of instances \instName{Cant1} and \instName{Cant2}. 
      The load forces are marked in dark gray.
    On the left the structure is fixed on the gray wall.}
    \label{fig:num_ex:cantilever}
\end{figure}
\FloatBarrier
\paragraph{Hook Problem}

The last problem, also described in~\cite{diss_hoheisel}, is referred to as the \emph{hook problem},
named for the shape of the structure, as shown in Figure~\ref{fig:num_ex:hook_potential}. 
This problem consists of \num{35} nodes, with the structure fixed at its top.
Two load forces are applied to the node at the far middle right in instances \instName{Hook1} and \instName{Hook2}. 
The load forces are taken from~\cite{diss_hoheisel}.
Furthermore a third force is applied in instances \instName{Hook3}, \instName{Hook4}, and \instName{Hook5}
at the far upper right corner as a linear combination of the two load forces in the far middle right corner.
\begin{figure}[t]
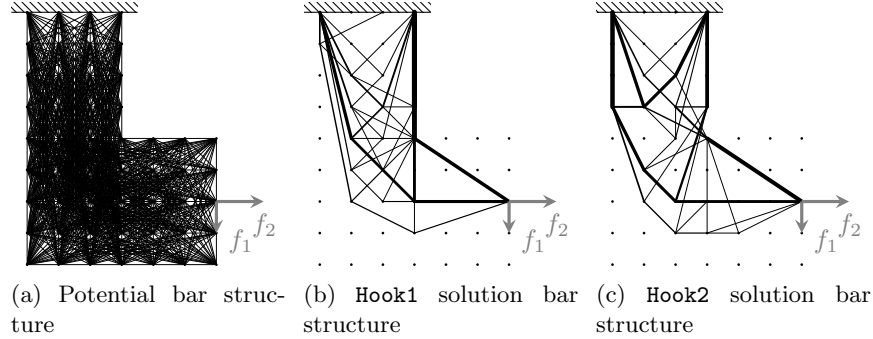

  \centering
  \begin{subfigure}[t]{0.3\textwidth}
        \includegraphics[width=\textwidth]{hook_potential}
        \caption{Potential bar structure}
        \label{fig:num_ex:hook_potential}
      \end{subfigure}~
      \begin{subfigure}[t]{0.3\textwidth}
        \includegraphics[width=\textwidth]{hook_1}
        \caption{\instName{Hook1} solution bar structure}
        \label{fig:num_ex:hook_solution_1}
      \end{subfigure}~
      \begin{subfigure}[t]{0.3\textwidth}
        \includegraphics[width=\textwidth]{hook_2}
        \caption{\instName{Hook2} solution bar structure}
        \label{fig:num_ex:hook_solution_2}
      \end{subfigure}
    \caption{Ground structure of the hook problem in (a) and solutions of the instance \instName{Hook1} in (b) and \instName{Hook2} in (c) which differ in the parameter $\overline{\sigma}$. 
    The load forces $f_1$ and $f_2$ are marked in dark gray. The structure is fixed on the gray area.}
    \label{fig:num_ex:hook_three}
\end{figure}
All nodes have integral coordinates, which allows for the detection of overlapping potential bars as in the previous case, 
leading to a total of \num{661} potential bars. As in the previous cantilever example we only show parameter sets, 
where actual vanished constraints are present. We therefore consider the three different parameter sets: 
$\overline{a}=100$, $c=100$, $(\overline{\sigma},\underline{\sigma})\in\{(3.5,-3.5),~(3.0,-3.0),~(3.0,-1.7)\}$. 
See Table~\ref{tab:num_ex:problemslist} for further details of all hook instances.
For all \instName{Hook} instances we used a tolerance of $10^{-10}$ in \textsc{Ipopt}~\cite{ipopt}.
The detailed solutions are listed in Table~\ref{tab:num_ex:problemssol} and
visually presented in Figure~\ref{fig:num_ex:hook} and Figure~\ref{fig:num_ex:hook_three}.
\begin{figure}[t]
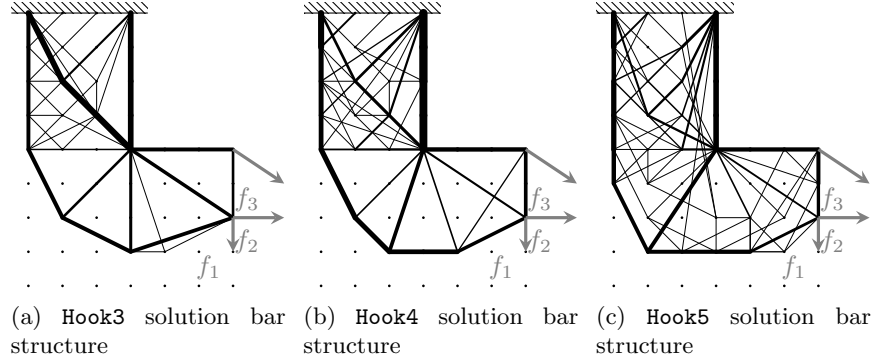

  \centering
  \begin{subfigure}[t]{0.3\textwidth}
    \includegraphics[width=\textwidth]{hook_3}
    \caption{\instName{Hook3} solution bar structure}
    \label{fig:num_ex:hook_solution_3}
  \end{subfigure}~
  \begin{subfigure}[t]{0.3\textwidth}
    \includegraphics[width=\textwidth]{hook_4}
    \caption{\instName{Hook4} solution bar structure}
    \label{fig:num_ex:hook_solution_4}
  \end{subfigure}~
  \begin{subfigure}[t]{0.3\textwidth}
    \includegraphics[width=\textwidth]{hook_5}
    \caption{\instName{Hook5} solution bar structure}
    \label{fig:num_ex:hook_solution_5}
  \end{subfigure}
  \caption{We show the solutions of the hook problem with three load forces $f_1$,~$f_2$ and $f_3$, which are marked in dark gray. 
  The solution of the instance \instName{Hook3} in (a) and \instName{Hook4} in (b) which differ in the parameter $\overline{\sigma}$ and $\underline{\sigma}$.
  The solution of the instance \instName{Hook5} differs to \instName{Hook4} in the value of $\bar{a}=0.7$. 
  The structure is fixed on the gray area.}
  \label{fig:num_ex:hook}
\end{figure}
In \instName{Hook1}, it can be observed that, in contrast to~\cite{diss_hoheisel},
the stress constraint is not active for the first load $l=1$. 
We observe more bars with $a_i>0$ in our solution resulting in a higher volume.
For the second load $l=2$ no stress constraint is active
with an overall maximal value of $\hat{\sigma}=2.6373$.
The structure forms a flat half-circle open at the top (see Figure~\ref{fig:num_ex:hook_solution_1}).
The structure is slim in the lower section, as in the single load case in~\cite{flow_preprint} and in
instance \instName{Hook2} in~\cite{diss_hoheisel}.
In the upper section, there are more (small) bars that serve to stabilize the structure. 
There are \num{49} bars in total. In contrast to the results presented in~\cite{diss_hoheisel}, we obtain a larger 
maximum cross-sectional area value, $\max a_i = 0.63861$, which again leads to a higher volume of \num{17.54340}.
The parameters $\rho=0.1$ and $\lambda_\mathrm{inc}=1.8$ were used.

Figure~\ref{fig:num_ex:hook_solution_2} shows the solution for instance \instName{Hook2}. 
A larger volume of \num{18.13786} is computed than in~\cite{diss_hoheisel} and \instName{Hook1}. 
The structure has been changed to adapt to the smaller $\overline{\sigma}=3.0$. 
This results in more vanishing constraints in the lower branch and increases the computation time. In general, 
both instances require more than \num{2000} \texttt{Ipopt} iterations and a significant amount of time to compute, 
which illustrates the complexity of vanishing constraint problems. 
For the instances \instName{Hook2} and \instName{Hook3} the parameters $\rho=1.0$ and $\lambda_\mathrm{inc}=1.8$ were used 

The new instance \instName{Hook3} has three load cases, where the third load acts on the far upper right corner. 
The applied force is the sum of the other two loading forces. Compared to \instName{Hook1} we take the same $\overline{\sigma}=3.5$. 
Surprisingly we obtain fewer constraints in their lower branches, but we obtain larger volume of $21.31727$. In Figure~\ref{fig:num_ex:hook_solution_3}
one can see many diagonal bars in the upper part producing this higher volume.
The $\max a_i = 0.7899$ has also increased compared to \instName{Hook1}. 
The number of \texttt{Ipopt} iterations decreased to \num{1977} compared to \instName{Hook1} and \instName{Hook2}. 
However, due to the increased problem size the computation time has increased.
The values $\hat{\sigma}_\mathrm{max}=3.036$ and $\hat{\sigma}_\mathrm{min}=-3.0683$ occur in the second load case, which means no stress constraint is active.

The next instance is \instName{Hook4}, which differs from \instName{Hook3} in the smaller stress constraint values 
$\overline{\sigma}=3.0$ and $\underline{\sigma}=1.7$ (see Table~\ref{tab:num_ex:problemslist}).
Since in \instName{Hook3} the maximum value of $\hat{\sigma}_\mathrm{max}=3.036$ is slightly above $\overline{\sigma}=3.0$, 
it is unsurprising that the solutions are quite similar. 
The maximal stress value is active in the second and third load case, while the minimal stress of \num{-1.7} is active in all three load cases. 
The number of actual vanished constraints are also increased to $58$ constraints in the lower branch.
The volume of \num{25.89200} is larger (see Table~\ref{tab:num_ex:problemssol}), since more bars appear and $\max a_i = 1.039$ has increased. 
To speed up the computation time, we used the values $\rho=10$ and $\lambda_\mathrm{inc}=2.1$. Although the number of iterations decreased to \num{17},
the number of \texttt{Ipopt} iterations increased to \num{3244} and the computation time to \num{22427.64} seconds.

\begin{table}[hb]
  \caption{Detailed solutions of listed problems from Table~\ref{tab:num_ex:problemslist} showing the optimal volume $V^\star$, the numbers of resulting bars,
  the maximal stress values $\overline{\sigma}$ and the maximal stress value for appearing bars $\hat{\sigma}_\mathrm{max}$ (\ie $a_i>0$) over all load cases,
  the configuration between lower branches ($\#\FS_0$) and upper branches ($\#\FS_+$), the number of iterations, the overall number of \texttt{Ipopt} iterations, 
  and the computation time in seconds. Instances initialized with the previous instance are marked with *.}
  \label{tab:num_ex:problemssol}
  \begin{adjustbox}{max width=\textwidth}
      \centering\begin{tabular}{l S[table-format=2.4] S[table-format=2] S[table-format=2.4] S[table-format=2.4]  r S[table-format=2] S[table-format=4] S[table-format=5.2]}
      \toprule
      \textbf{Problem} & \textbf{$V^\star$} & \textbf{\#bars} & $\sigma_{\max}$ & $\hat{\sigma}_{\max}$ &  $\#\FS_0 | \#\FS_+$ & \textbf{\#it} & \textbf{\texttt{Ipopt}\#it} & {\textbf{time} [\si{\second}]}\\
      \midrule
      \instName{TenBar1}&   9.00000 &  5 & 1.0     & 1.0 & $0|20$\,~~~ & 10 & 258 & 0.90\\ %rho 1e-2, inc = 1.8
      \instName{TenBar2}&   9.44323 &  7 & 1.0     & 1.0 & $0|30$\,~~~ & 12 & 344 & 1.55\\ %rho 1e-2, inc = 1.8
      \instName{Cant1} &    23.88549 & 30 &  3.1262 & 2.2 & $16|432$~~ & 14 & 1497 &604.94\\ %rho 1e-3, inc = 1.9
      \instName{Cant2}* &   33.68315 & 38 &  5.1518 & 2.2 & $52|396$~~ & 13 & 1266 &528.41\\ %rho 1e-3, inc = 1.9
      \instName{Hook1}&    17.54340 & 49 & 18.5920 & 2.6373 & $37|1285$~ & 18 & 2062 &11022.07\\ %rho 1e-1, inc = 1.8
      \instName{Hook2}&    18.13786 & 47 & 13.9829 & 3.0 & $36|1286$~ & 19 & 2464 & 13176.11\\  %rho 1e0, inc = 1.8
      \instName{Hook3}&    21.31727 & 53 &  8.9399 & 3.0360 & $19|1964$~ & 18 & 1977 &14439.54\\ %rho 1e0, inc = 1.8
      \instName{Hook4}&    25.89200 & 57 &  7.0122 & 3.0 & $58|1950$~ & 17 & 3244 & 22427.64\\ %rho = 1e1, lambda_inc = 2.1
      \instName{Hook5}&    27.43300 & 83 &  7.5173 & 3.0 & $96|1887$~ & 17 & 5518 &  39996.46\\ %rho = 5e1, lambda_inc = 2.2
      \bottomrule
    \end{tabular}
  \end{adjustbox}
\end{table}
\newpage
The final instance is \instName{Hook5} was configured with $\bar{a}=0.7$, based on the maximum value $\max a_i = 0.7899$ from instance \instName{Hook3}.
The values $\rho=50$ and $\lambda_\mathrm{inc}=2.2$ were used to keep the iterations at a low level. 
However the \texttt{Ipopt} iterations increased to \num{5518} and the resulting longer computation time of \num{39996.46} seconds 
illustrating the complexity of the instance. The structure is shown in Figure~\ref{fig:num_ex:hook_solution_5} with 
\num{83} bars and a volume of \num{27.43300}. 
In the lower part more bars are present to support the upper load force $f_3$.

\FloatBarrier
\paragraph{Comparison to Relaxation Approach}
To the best of our knowledge, the relaxation method proposed in~\cite{diss_pal} is the only available approach for solving MPBVC problems. 
We therefore briefly compare results of the relaxation framework in~\cite{diss_pal} for the truss topology instances.
For truss topology optimization problems, the relaxed formulation reads
\begin{equation*}
  \begin{aligned}
      \min_{a \in \Real^{N}, u \in \Real^{Ld}} \quad & \sum_{i = 1}^{N} \ell_{i} a _i \\
    \st \quad & K(a) u_{l} = f_{l} & \quad \text{for all } l = 1, \ldots , L \\
                                                     & f_{l}^{T} u_{l} \leq c & \quad \text{for all } l = 1, \ldots , L \\
                                                     & 0 \leq a_{i} \leq \overline{a} & \quad \text{for all } i = 1, \ldots , N \\
                                                     & r_i(a,u)a_{i} \geq -t & \quad \text{for all } i = 1, \ldots , N.
  \end{aligned}
\end{equation*} Where $t>0$ is a \emph{relaxation parameter} and \begin{equation*}
  \begin{aligned}
    r^{i}(a,u) &\define \min_{l=1,\dots,L} (-\sigma_{il}^2(a,u) +(\overline{\sigma}+\underline{\sigma})\sigma_{il}(a,u) - \underline{\sigma}\overline{\sigma})
  \end{aligned}
\end{equation*}
 or a smooth approximation of $r_i$.
The computation times are very short whenever a solution is obtained, \ie whenever \texttt{Ipopt}~\cite{ipopt} terminates with the message "Optimal solution found".
Convergence is achieved only for the instances \instName{TenBar1} and \instName{TenBar2} for all relaxation parameters $t \in \{10^{-2},10^{-3},\dots,10^{-9}\}$, 
resulting in the same structures with the same volumes as the solutions presented in Figure~\ref{fig:num_ex:tenbar} and Table~\ref{tab:num_ex:problemssol}. 
Independent of $t$, the solver \texttt{Ipopt} required $16$ iterations for \instName{Tenbar1} for convergence, 
while $8$ iterations for \instName{TenBar2} were performed for convergence.
Note that convergence was obtained by using the internal \texttt{CasADi} functions \texttt{mmin} or \texttt{logsumexp}. Using \texttt{logsumexp} results in an underestimation 
of the volumes, lower than in the single load case stated in~\cite{flow_preprint}.

For all remaining instances, we tested relaxation parameters $t = 10$ down to $t = 10^{-9}$, combined with either zero initialization, 
the initialization described in~\cite{flow_preprint} or results of the previous larger relaxation parameter. In all cases, the solver terminated at
infeasible points or reached the maximum number of iterations. Increasing the maximum number of \texttt{Ipopt} iterations to \num{5000}
did not alter the results, since \texttt{Ipopt} consistently struggled to sufficiently reduce the dual infeasibility.
\FloatBarrier

\section{Conclusion}
This work presents a novel method of a gradient flow approach for 
computing strongly stationary points of Mathematical Programs with Blocks of Vanishing Constraints.
Due to the natural loss of constraint qualifications for MPBVC we developed a corresponding constraint qualification theory
by extending the existing theory for classical Mathematical Problems with Vanishing Constraints.
We proved that an LICQ-type constraint qualification guarantees strong stationarity of local minimizers.
The proposed method exploits the problem's structure via a slack variable reformulation of the vanishing and controlling constraints.
In particular, it selects and transitions between convex subsets of the feasible region associated with the slack variables 
corresponding to each pair of vanishing and controlling constraints. 
The strategy enables tracking of the primal-dual gradient/anti-gradient flow at each iteration through the solution of a reduced subproblem. 
The performance and effectiveness of the proposed method is demonstrated on truss topology design problems 
with multiple load cases. It outperforms a state-of-the-art relaxation method used for comparison.

\section*{Acknowledgements}

Supported by the German Federal Ministry of Research, Technology and Space (BMFTR) under Grant No. 05M22VHA.
{\sloppy
\printbibliography

@article{mpvc,
  author = {Achtziger, Wolfgang and Kanzow, Christian},
  title = {Mathematical programs with vanishing constraints: optimality conditions and constraint qualifications},
  journal = {Mathematical Programming},
  year = {2008},
  volume = {114},
  pages = {69--99},
  publisher = {Springer},
  doi = {10.1007/s10107-006-0083-3}
}

@article{semismooth,
  author = {Mifflin, Robert},
  title = {Semismooth and semiconvex functions in constrained optimization},
  journal = {SIAM Journal on Control and Optimization},
  year = {1977},
  volume = {15},
  number = {6},
  pages = {959--972},
  publisher = {SIAM},
  doi = {10.1137/0315061}
}

@article{ssn,
  author = {Hinterm{\"u}ller, Michael},
  title = {Semismooth Newton methods and applications},
  journal = {Department of Mathematics, Humboldt-University of Berlin},
  year = {2010}
}

@book{mpec,
  author = {Luo, Zhi-Quan and Pang, Jong-Shi and Ralph, Daniel},
  title = {Mathematical programs with equilibrium constraints},
  publisher = {Cambridge University Press},
  year = {1996},
  isbn = {978-0521-572-903}
}

@article{mfcq,
  author = {Mangasarian, Olvi and Fromovitz, Stan},
  title = {The Fritz John necessary optimality conditions in the presence of equality and inequality constraints},
  journal = {Journal of Mathematical Analysis and applications},
  year = {1967},
  volume = {17},
  number = {1},
  pages = {37--47},
  publisher = {Elsevier},
  doi = {10.1016/0022-247X(67)90163-1}
}

@misc{flow_preprint,
  author = {Christoph Hansknecht and Julian Niederer and Andreas Potschka},
  title = {A Flow-based Method for Problems with Vanishing Constraints},
  year = {2026},
  url = {https://arxiv.org/abs/2602.23830},
  eprint = {2602.23830},
  archiveprefix = {arXiv},
  primaryclass = {math.OC}
}

@article{sequential_homotopy,
  author = {Potschka, Andreas and Bock, Hans Georg},
  title = {A sequential homotopy method for mathematical programming problems},
  journal = {Mathematical Programming},
  year = {2021},
  volume = {187},
  number = {1},
  pages = {459--486},
  publisher = {Springer},
  doi = {10.1007/s10107-020-01488-z}
}

@book{numerical_optimization,
  author = {Nocedal, Jorge and Wright, Stephen J.},
  title = {Numerical Optimization},
  publisher = {Springer New York},
  year = {2006},
  isbn = {978-0-387-30303-1}
}

@article{scholtes,
  author = {Scholtes, Stefan},
  title = {Convergence properties of a regularization scheme for mathematical programs with complementarity constraints},
  journal = {SIAM Journal on Optimization},
  year = {2001},
  volume = {11},
  number = {4},
  pages = {918--936},
  publisher = {SIAM},
  doi = {10.1137/S1052623499361233}
}

@article{mpvc_stat_cons,
  author = {Hoheisel, Tim and Kanzow, Christian},
  title = {Stationary conditions for mathematical programs with vanishing constraints using weak constraint qualifications},
  journal = {Journal of Mathematical Analysis and Applications},
  year = {2008},
  volume = {337},
  number = {1},
  pages = {292--310},
  publisher = {Elsevier}
}

@article{guignard_stat_cons,
  author = {Guignard, Monique},
  title = {Generalized Kuhn--Tucker conditions for mathematical programming problems in a Banach space},
  journal = {SIAM Journal on Control},
  year = {1969},
  volume = {7},
  number = {2},
  pages = {232--241},
  publisher = {SIAM},
  doi = {10.1137/0307016}
}

@techreport{abadie1966kuhn,
  author = {Abadie, Jean},
  title = {On the Kuhn-Tucker theorem},
  year = {1966},
  url = {https://api.semanticscholar.org/CorpusID:122112088}
}

@article{on_the_abadie,
  author = {Hoheisel, Tim and Kanzow, Christian},
  title = {On the Abadie and Guignard constraint qualifications for mathematical programmes with vanishing constraints},
  journal = {Optimization},
  year = {2009},
  volume = {58},
  number = {4},
  pages = {431--448},
  publisher = {Taylor \& Francis},
  doi = {10.1080/02331930701763405}
}

@phdthesis{diss_hoheisel,
  title={Mathematical programs with vanishing constraints},
  author={Hoheisel, Tim},
  year={2009},
  school={University of W{\"u}rzburg}
}

@book{topology_optimization,
  author = {Bendsoe, Martin Philip and Sigmund, Ole},
  title = {Topology optimization: theory, methods, and applications},
  publisher = {Springer Science \& Business Media},
  year = {2013},
  isbn = {978-3642076985}
}

@article{Hoheisel22,
  author = {Tim Hoheisel and Blanca Pablos and Aram Pooladian and Alexandra Schwartz and Luke Steverango},
  title = {A study of one-parameter regularization methods for mathematical programs with vanishing constraints},
  journal = {Optimization Methods and Software},
  year = {2022},
  volume = {37},
  number = {2},
  pages = {503--545},
  publisher = {Taylor \& Francis},
  doi = {10.1080/10556788.2020.1797025},
}

@article{mpbvc_pal,
  author = {Palagachev, Konstantin and Gerdts, Matthias},
  title = {Mathematical programs with blocks of vanishing constraints arising in discretized mixed-integer optimal control problems},
  journal = {Set-Valued and Variational Analysis},
  year = {2015},
  volume = {23},
  number = {1},
  pages = {149--167},
  publisher = {Springer},
  doi = {10.1007/s11228-014-0297-0}
}

@phdthesis{diss_pal,
  author = {Palagachev, Konstantin},
  title = {Mixed-integer optimal control and bilevel optimization: Vanishing constraints and scheduling tasks},
  year = {2017},
  school = {Neubiberg, Universit{\"a}t der Bundeswehr M{\"u}nchen, 2017}
}

@article{relaxation_method,
  author = {Izmailov, Alexey and Solodov, Mikhail},
  title = {Mathematical programs with vanishing constraints: optimality conditions, sensitivity, and a relaxation method},
  journal = {Journal of Optimization Theory and Applications},
  year = {2009},
  volume = {142},
  number = {3},
  pages = {501--532},
  publisher = {Springer},
  doi = {10.1007/s10957-009-9517-4}
}

@article{giorgi,
  author = {Giorgi, Giorgio and others},
  title = {A guided tour in constraint qualifications for nonlinear programming under differentiability assumptions},
  journal = {University of Pavia, Department of Economics and Management, Tech. Rep},
  year = {2018}
}

@article{ipopt,
  author = {W{\"a}chter, Andreas and Biegler, Lorenz T},
  title = {On the implementation of an interior-point filter line-search algorithm for large-scale nonlinear programming},
  journal = {Mathematical programming},
  volume = {106},
  number = {1},
  pages = {25--57},
  year = {2006},
  publisher = {Springer},
  doi = {10.1007/s10107-004-0559-y}
}

@article{casadi,
  author = {Joel A E Andersson and Joris Gillis and Greg Horn and James B Rawlings and Moritz Diehl},
  title = {{CasADi} -- {A} software framework for nonlinear optimization and optimal control},
  journal = {Mathematical Programming Computation},
  volume = {11},
  number = {1},
  pages = {1--36},
  year = {2019},
  publisher = {Springer},
  doi = {10.1007/s12532-018-0139-4}
}

@article{farkas,
  author = {Farkas, Julius},
  title = {Theorie der einfachen Ungleichungen.},
  journal = {Journal f{\"u}r die reine und angewandte Mathematik (Crelles Journal)},
  year = {1902},
  volume = {1902},
  number = {124},
  pages = {1--27},
  publisher = {De Gruyter},
  url = {http://eudml.org/doc/149129}
}

@misc{nurkanovic2025globally,
  author = {Nurkanovi{\'c}, Armin and Leyffer, Sven},
  title = {A globally convergent method for computing B-stationary points of mathematical programs with equilibrium constraints},
  year = {2025},
  url = {https://arxiv.org/abs/2501.13835},
  eprint = {2501.13835},
  archiveprefix = {arXiv},
  primaryclass = {math.OC}
}

@article{steffensen2010new,
  author = {Steffensen, Sonja and Ulbrich, Michael},
  title = {A new relaxation scheme for mathematical programs with equilibrium constraints},
  journal = {SIAM Journal on Optimization},
  year = {2010},
  volume = {20},
  number = {5},
  pages = {2504--2539},
  publisher = {SIAM},
  doi = {10.1137/090748883}
}

@article{nurkanovic2025solving,
  author = {Nurkanovi{\'c}, Armin and Pozharskiy, Anton and Diehl, Moritz},
  title = {Solving Mathematical Programs with Complementarity Constraints Arising in Nonsmooth Optimal Control},
  journal = {Vietnam Journal of Mathematics},
  year = {2025},
  volume = {53},
  number = {3},
  pages = {659--697},
  publisher = {Springer},
  doi = {10.1007/s10013-024-00704-z}
}

@article{hoheisel2013theoretical,
  author = {Hoheisel, Tim and Kanzow, Christian and Schwartz, Alexandra},
  title = {Theoretical and numerical comparison of relaxation methods for mathematical programs with complementarity constraints},
  journal = {Mathematical Programming},
  year = {2013},
  volume = {137},
  number = {1},
  pages = {257--288},
  publisher = {Springer},
  doi = {10.1007/s10107-011-0488-5}
}

@article{FP2010,
  author = {Diego Feijer and Fernando Paganini},
  title = {Stability of primal-dual gradient dynamics and applications to network optimization},
  journal = {Automatica},
  year = {2010},
  volume = {46},
  number = {12},
  pages = {1974-1981},
  doi = {10.1016/j.automatica.2010.08.011}
}

@article{CMC2016,
  author = {Ashish Cherukuri and Enrique Mallada and Jorge Cortés},
  title = {Asymptotic convergence of constrained primal-dual dynamics},
  journal = {Systems \& Control Letters},
  year = {2016},
  volume = {87},
  pages = {10-15},
  doi = {10.1016/j.sysconle.2015.10.006}
}

@article{QuLi2018,
  author  = {Qu, Guannan and Li, Na},
  title   = {On the Exponential Stability of Primal-Dual Gradient Dynamics},
  journal = {IEEE Control Systems Letters},
  volume  = {3},
  number  = {1},
  pages   = {43-48},
  year    = {2019},
  doi     = {10.1109/LCSYS.2018.2851375}
}

@article{CP2011,
  author  = {Antonin Chambolle and Thomas Pock},
  title   = {A First-Order Primal-Dual Algorithm for Convex Problems with Applications to Imaging},
  journal = {Journal of Mathematical Imaging and Vision},
  volume  = {40},
  number  = {1},
  pages   = {120--145},
  year    = {2011},
  doi     = {10.1007/s10851-010-0251-1}
}
}
%\printbibliography

\appendix
\section{Appendix}
\label{sec:appendix}
The proof of the following theorem is standard and is based on ~\cite{mpbvc_pal,diss_pal}
\vanishingKKT*
\begin{proof}

To prove the theorem we use Farkas Lemma in a modern reformulation of the original result \cite{farkas}:
\begin{lemma}[Farkas]\label{vanishing: lem: Farkas}
Let $B\in\Real^{m\times n}$ and $c\in \Real^n$. Then the following statements are equivalent:\begin{itemize}
    \item The system $B^T\xi =c$ has a solution $\xi\geq0$.
    \item $c^Td\geq0$ holds for all elements in the set $\{d\in\Real^n\,\mid\,Bd\geq0\}$.
\end{itemize}
\end{lemma}

Since $x^\star$ is a local minimum it holds $\nabla f(x^\star)^T\tilde{d}\geq0$ for all $\tilde{d}\in\mathcal{T}(X,x^\star)$,
\ie with \ref{word:GCQ} if follows that $- \nabla f(x^\star)\in \mathcal{T}(X,x^\star)^{-}= \Lag(x^\star)^{-}$.
From the definition of the polar cone it then follows that
$\nabla f(x^\star)^Td \geq $ for all $d\in\Lag(x^\star)$,
where $\Lag(x^\star)$ is given according to Lemma~\ref{sec:linearized_cone_def}.
Based on the matrix $A$ defined as
\begin{align*}
  A := \begin{pmatrix}
    -\DECon(x^\star)^T \\
    \DECon(x^\star)^T \\
    \DICon_{\I\ICon(x^\star)}(x^\star)^T \\
    -\DConCon{\II{0}{-}^\ConConUP(x^\star)}(x^\star)^T  \\
    \DConCon{\I{0}^\ConConUP(x^\star)}(x^\star)^T  \\
    \nabla \VanConUP_{\II{+}{0}(x^\star)}(x^\star)^T \\
  \end{pmatrix}
\end{align*}
this condition is equivalent to $\nabla f(x)^Td \geq 0$ for all $d$ such that $Ad \geq 0$.
    It follows from Lemma~\ref{vanishing: lem: Farkas}
    that there exists a vector $\xi$ such that $\nabla f(x) - A^T\xi = 0$, where $\xi\in\Real^N$
    with $N= 2\cdot\#\SetN_\ECon+\#\I\ICon(x^\star) + 2\cdot\#\II{0}{-}(x^\star) + \#\II{0}{0}(x^\star) + \#\II{0}{+}(x^\star) + \#\II{+}{0}(x^\star)$.
    We partition $\xi$ according to the rows of $A$ and obtain
    \begin{align*}
        \xi = (\hat\mu^\ECon_-,\hat\mu^\ECon_+,\hat\mu^\ICon,\hat\mu^\ConConUP_-,\hat\mu^\ConConUP,\hat\mu^\VanConUP)\geq 0.
    \end{align*} 
    We set \begin{align*}
    	\mu^\ECon_k &:= \hat\mu^\ECon_{k,+}-\hat\mu^\ECon_{k,+}\quad\quad\!\text{for }k\in\SetN_\ECon, \\
    	\mu^\ICon_l &:= \begin{cases}
    		\hat\mu^\ICon_l &\quad\quad\quad\quad\text{for }l\in\I\ICon(x^\star),\\
    		0 &\quad\quad\quad\quad\text{otherwise},
    	\end{cases}\\
        \mu^\ConConUP_i &:= \begin{cases}
            \hat\mu^\ConConUP_i-\hat\mu^\ConConUP_{i,-} &\text{ if }i\in\II{0}{-}^\ConConUP(x^\star),\\
            \hat\mu^\ConConUP_i &\text{ if }i\in\I{0}^\ConConUP(x^\star) \text{ and } i\not\in\II{0}{-}^\ConConUP(x^\star),\\
            0 &\text{ for }i\in\II{+}{0}^\ConConUP(x^\star)\cup\II{+}{+}^\ConConUP(x^\star),
        \end{cases}\\
        \mu^\VanConUP_{ij}&:=\begin{cases}
            \hat\mu^\VanConUP_{ij} &\quad\quad\quad\quad\text{for }(i,j)\in\II{+}{0}(x^\star),\\
            0 &\quad\quad\quad\quad\text{otherwise},
        \end{cases}
    \end{align*}
    and obtain the claimed result.
\end{proof}
The following theorem is based on~\cite[Theorem~1]{mpvc}.
\vanishingEquiv*

\begin{proof}
  Since the Lagrange multipliers associated with the equality and inequality
  constraints remain unchanged, we only consider the block vanishing constraints.
  We first show that the classical KKT conditions imply strong stationarity under GCQ.
    Since~\ref{word:GCQ} holds at the local minimum $x^\star$ there exists Lagrange multipliers $\lambda^\ConConUP\in\Real^{n_\ConConUP}$ and $\lambda^{\VanConUP\ConConUP}\in\Real^{n_\VanConUP\cdot n_\ConConUP}$
    such that \begin{align}\label{vanishing: kkt-lagrange-equation-classic}
        \nabla f(x^\star) -\sum_{i\in\SetN_\ConConUP}\lambda^\ConConUP_i\DConCon i(x^\star)-\sum_{(i,j)\in\SetN_\ConConUP\times\SetN_\VanConUP} \lambda^{\VanConUP\ConConUP}_{ij}\nabla \big(\ConCon i(x^\star)\VanCon ij(x^\star)\big) = 0
    \end{align} with \begin{align}
       \ConCon i(x^\star)\VanCon ij(x^\star)&\geq 0, & \ConCon i(x^\star)&\geq 0, \notag\\
        \lambda^{\VanConUP\ConConUP}_{ij}\ConCon i(x^\star)\VanCon ij(x^\star)&= 0, & \lambda^\ConConUP_i\ConCon i(x^\star)&=0,\label{vanishing: kkt-lagrange-multiplier-classic} \\
        \lambda^{\VanConUP\ConConUP}_{ij}&\geq0, & \lambda^{h}_{i}&\geq0,\notag
    \end{align}where $i\in\SetN_\ConConUP$ and $(i,j)\in\SetN_\ConConUP\times\SetN_\VanConUP$.
    The gradient of the product is \begin{align*}
        \nabla \big(\ConCon i(x^\star)\VanCon ij(x^\star)\big) = \begin{cases}
            0, &\hspace{-9em} \text{ if }(i,j)\in\II{0}{0},\\
           \ConCon i(x^\star)\nabla \VanCon ij(x^\star), &\hspace{-9em}\text{ if }(i,j)\in\II{+}{0}, \\
            \DConCon i(x^\star)\VanCon ij(x^\star), &\hspace{-9em}\text{ if }(i,j)\in\II{0}{+}\cup\II{0}{-}, \\
            \DConCon i(x^\star)\VanCon ij(x^\star) +\ConCon i(x^\star)\nabla \VanCon ij(x^\star), &\hspace{-0.75em}\text{ if }(i,j)\in\II{+}{+}.\\
        \end{cases}
    \end{align*}
    We introduce the multipliers \begin{align}\label{vanishing: construction-lagrange-multiplier}
        \mu^\ConConUP_i &= \lambda^\ConConUP_i + \sum_{j\in\SetN_\VanConUP}\lambda^{\VanConUP\ConConUP}_{ij}\VanCon ij(x^\star) & \text{and} && \mu^\VanConUP_{ij} =\lambda^{\VanConUP\ConConUP}_{ij}\ConCon i(x^\star).
    \end{align}
    We now verify the conditions~\eqref{vanishing: kkt-lagrange-equation} and~\eqref{vanishing: kkt-lagrange-multiplier}.
    \begin{align*}
    &\nabla_xL(x,\mu^\ConConUP,\mu^\VanConUP)\\
        &= \nabla f(x^\star) -\sum_{i\in\SetN_\ConConUP}\mu^\ConConUP_i\DConCon i(x^\star)-\sum_{(i,j)\in\SetN_\ConConUP\times\SetN_\VanConUP} \mu^{g}_{ij}\nabla \VanCon ij(x^\star)\\
        &= \nabla f(x^\star) -\sum_{i\in\SetN_\ConConUP}\big[\lambda^\ConConUP_i+\sum_{j\in\SetN_\VanConUP}\lambda^{\VanConUP\ConConUP}_{ij}\VanCon ij(x^\star)\big]\DConCon i(x^\star)\\
        &\quad-\sum_{(i,j)\in\SetN_\ConConUP\times\SetN_\VanConUP} \lambda^{\VanConUP\ConConUP}_{ij}\big(\ConCon i(x^\star)\DVanCon ij(x^\star)\big)\\
        &= \nabla f(x^\star) -\sum_{i\in\SetN_\ConConUP}\lambda^\ConConUP_i\DConCon i(x^\star)\\
        &\quad - \sum_{(i,j)\in\SetN_\ConConUP\times\SetN_\VanConUP}\lambda^{\VanConUP\ConConUP}_{ij}\VanCon ij(x^\star)\DConCon i(x^\star) -\sum_{(i,j)\in\SetN_\ConConUP\times\SetN_\VanConUP} \lambda^{\VanConUP\ConConUP}_{ij}\big(\ConCon i(x^\star)\DVanCon ij(x^\star)\big)\\
        &= \nabla f(x^\star) -\sum_{i\in\SetN_\ConConUP}\lambda^\ConConUP_i\DConCon i(x^\star)-\sum_{(i,j)\in\SetN_\ConConUP\times\SetN_\VanConUP} \lambda^{\VanConUP\ConConUP}_{ij}\nabla \big(\ConCon i(x^\star)\VanCon ij(x^\star)\big) \\&= 0.
    \end{align*}
    By construction from~\eqref{vanishing: construction-lagrange-multiplier} it follows that \begin{align*}
        \mu^\ConConUP_i &=\lambda^\ConConUP_i + \sum_{j\in\SetN_\VanConUP}\lambda^{\VanConUP\ConConUP}_{ij}\VanCon ij(x^\star) = \lambda^\ConConUP_i, \text{ since } \lambda^{\VanConUP\ConConUP}_{ij}\VanCon ij(x^\star)=0,\\
        \mu^\VanConUP_{ij} &= \lambda^{\VanConUP\ConConUP}_{ij}\ConCon i(x^\star) \begin{cases}
            \geq 0 &\text{ if } (i,j)\in\II{+}{0} \\
            =0 & \text{otherwise}
        \end{cases}
    \end{align*}
    Conversely, since~\ref{word:GCQ} holds we know that the local minimum $x^\star$ satisfies~\eqref{vanishing: kkt-lagrange-equation} 
    and~\eqref{vanishing: kkt-lagrange-multiplier}
    with the Lagrange multipliers $\mu^\ConConUP\in\Real^{n_\ConConUP}$ and $\mu^\VanConUP\in\Real^{n_\VanConUP\cdot n_\ConConUP}$.
    We set \begin{align*}
        \lambda^{\VanConUP\ConConUP}_{ij} &= \frac{\mu^\VanConUP_{ij}}{\ConCon i(x^\star)} &&\text{if }(i,j)\in\II{+}{0}, \\
        \lambda^{\VanConUP\ConConUP}_{ij} &= 0 &&\text{if }(i,j)\in\II{+}{+}, \\
        \lambda^{\VanConUP\ConConUP}_{ij} &\geq \max\bigg\{0,\frac{\mu^\ConConUP_i}{\VanCon ij(x^\star)}\bigg\} &&\text{if }(i,j)\in\II{0}{-}, \\
        \lambda^{\VanConUP\ConConUP}_{ij} &\in\bigg[0, \frac{\mu^\ConConUP_i}{\VanCon ij(x^\star)}\bigg]\neq\emptyset &&\text{if }(i,j)\in\II{0}{+}, \\
        \lambda^{\VanConUP\ConConUP}_{ij} &\geq 0  &&\text{if }(i,j)\in\II{0}{0}, \\
        &\text{and} &&\\
        \lambda^\ConConUP_i &= \mu^\ConConUP_i - \sum_{j\in\SetN_\VanConUP}\lambda^{\VanConUP\ConConUP}_{ij}\VanCon ij(x^\star). &&
    \end{align*}
For~\eqref{vanishing: kkt-lagrange-equation-classic} if follows that
\begin{align*}
    &\vspace{-1.5em}\nabla f(x^\star) -\sum_{i\in\SetN_\ConConUP}\lambda^\ConConUP_i\DConCon i(x^\star)-\sum_{(i,j)\in\SetN_\ConConUP\times\SetN_\VanConUP} \lambda^{\VanConUP\ConConUP}_{ij}\nabla \big(\ConCon i(x^\star)\VanCon ij(x^\star)\big) \\
    &=\nabla f(x^\star) -\sum_{i\in\SetN_\ConConUP}\big[\mu^\ConConUP_i - \sum_{j\in\SetN_\VanConUP}\lambda^{\VanConUP\ConConUP}_{ij}\VanCon ij(x^\star)\big]\DConCon i(x^\star)\\
    &\quad-\sum_{(i,j)\in\SetN_\ConConUP\times\SetN_\VanConUP} \lambda^{\VanConUP\ConConUP}_{ij}\nabla \big(\ConCon i(x^\star)\VanCon ij(x^\star)\big) \\
    &=\nabla f(x^\star) -\sum_{i\in\SetN_\ConConUP}\mu^\ConConUP_i\DConCon i(x^\star) 
     +\sum_{i\in\SetN_\ConConUP}\sum_{j\in\SetN_\VanConUP}\lambda^{\VanConUP\ConConUP}_{ij}\VanCon ij(x^\star)\DConCon i(x^\star)\\&\quad-\sum_{(i,j)\in\SetN_\ConConUP\times\SetN_\VanConUP}\lambda^{\VanConUP\ConConUP}_{ij}\VanCon ij(x^\star)\DConCon i(x^\star)-\sum_{(i,j)\in\SetN_\ConConUP\times\SetN_\VanConUP} \lambda^{\VanConUP\ConConUP}_{ij}\ConCon i(x^\star)\DVanCon ij(x^\star)\\
    &=\nabla f(x^\star) -\sum_{i\in\SetN_\ConConUP}\mu^\ConConUP_i\DConCon i(x^\star) -\sum_{(i,j)\in\II{+}{0}} \frac{\mu^\VanConUP_{ij}}{\ConCon i(x^\star)}\cdot\ConCon i(x^\star)\DVanCon ij(x^\star)\\
    &=\nabla f(x^\star) -\sum_{i\in\SetN_\ConConUP}\mu^\ConConUP_i\DConCon i(x^\star)-\sum_{(i,j)\in\II{+}{0}} \mu^\VanConUP_{ij}\DVanCon ij(x^\star)\\
    &=\nabla f(x^\star) -\sum_{i\in\SetN_\ConConUP}\mu^\ConConUP_i\DConCon i(x^\star)-\sum_{(i,j)\in\SetN_\ConConUP\times\SetN_\VanConUP} \mu^\VanConUP_{ij}\DVanCon ij(x^\star)\\
    &=0.
\end{align*}
By construction it holds that $\lambda^{\VanConUP\ConConUP}_{ij}\geq0$ and $\lambda_i\geq0$ for all $(i,j)\in\SetN_\ConConUP\times\SetN_\VanConUP$. The equations in~\eqref{vanishing: kkt-lagrange-multiplier-classic} are therefore satisfied.
\end{proof} 
\end{document}